\documentclass[a4paper,11pt,onecolumn,twoside]{article}
\usepackage{tikz}
\usetikzlibrary{arrows.meta}
\usepackage{fancyhdr}
\usepackage{amsmath}
\usepackage{hyperref} 
\sffamily 
\usepackage{bbm}
\usepackage[utf8]{inputenc}
\usepackage{amsmath,amsfonts,amssymb}
\usepackage{graphicx}
\usepackage{mathptmx}
\DeclareUnicodeCharacter{FB01}{fi}
\usepackage{amsthm}
\usepackage{booktabs}
\usepackage{amssymb}
\usepackage[mathscr]{eucal}
\usepackage[labelfont=bf]{caption}
\usepackage{indentfirst}
\usepackage{caption}
\usepackage{enumitem}
\usepackage{subfigure}
\usepackage{authblk}
\usepackage{natbib}
\usepackage[all]{xy}
\usepackage{geometry}
\usepackage{mathrsfs}
\usepackage{newtxmath}
\usepackage{hyperref}
\numberwithin{equation}{section}

\newtheorem{thm}{Theorem}[section]

\newtheorem{definition}[thm]{Definition}
\newtheorem{exm}[thm]{Example}

\newtheorem{lemma}[thm]{Lemma}
\newtheorem{cor}[thm]{Corollary}
\newtheorem{prop}[thm]{Proposition}

\newtheorem*{theorema}{Theorem A}
\newtheorem*{theoremb}{Theorem B}
\newtheorem*{theoremc}{Theorem C}
\newtheorem{rmk}[thm]{Remark}

\newcommand{\GG}{{^{G}_{G}\mathcal{YD}^{\Phi}}}

\newcommand{\GGA}{{^{G}_{G}\mathcal{YD}}}
\newcommand{\AAC}{{^{A}_{A}\mathcal{YD}(\mathcal{C})}}
\newcommand{\AARC}{{\mathcal{YD}^A_A}(\mathcal{C})}
\newcommand{\BBRC}{{\mathcal{YD}^B_B}(\mathcal{C})}
\newcommand{\BBC}{{^{B}_{B}\mathcal{YD}(\mathcal{C})}}

\newcommand{\HH}{{^{H}_{H}\mathcal{YD}}}
\newcommand{\BNG}{{^{\AN}_{\AN}\mathcal{YD}}}
\newcommand{\AMN}{\mathcal{A}(M\oplus N)}
\newcommand{\AN}{\mathcal{A}(N)}
\newcommand{\AMI}{{^{A(M_i)}_{A(M_i)}\mathcal{YD}}}
\newcommand{\ANI}{{^{A(N_i)}_{A(N_i)}\mathcal{YD}}}
\newcommand{\AMID}{{^{A(M_i^*)}_{A(M_i^*)}\mathcal{YD}}}

\newcommand{\VRn}{V^{\otimes \overrightarrow{n}}}

\newcommand{\NN}{\mathbb{N}_{0}}
\newcommand{\ZZ}{\mathbb{Z}}
\newcommand{\Gr}{\text{-Gr}}
\newcommand{\Mod}{\mathcal{M}}

\newcommand{\DG}{D^{\Phi}(G)}

\newcommand{\Zgr}{\mathbb{Z}\operatorname{-Gr}\mathcal{M}_{\mathbbm{k}}}
\newcommand{\BMN}{\mathcal{B}(M\oplus N)}

\newcommand{\bVG}{\textbf{Vec}_G^{\Phi}}

\newcommand{\HHp}{{^{H}_{H}\mathcal{YD}^{\phi}}}

\newcommand{\bfo}{\textbf{1}}
\newcommand{\BV}{\mathcal{B}(V)}
\newcommand{\BN}{\mathcal{B}(N)}

\usetikzlibrary{knots}
\newcommand{\kG}{\mathbbm{k}G}
\newcommand{\bB}{\mathcal{B}}

\newcommand{\wPhi}{\widetilde{\Phi}}

\hypersetup{
  colorlinks=true,
  linkcolor=magenta,
  citecolor=blue,
  filecolor=blue,
  urlcolor=blue
}

\title{\textbf{  Reflection of Nichols Algebras over Coquasi-Hopf Algebras }}
\author{Bowen Li and Gongxiang Liu
	}
\date{}
\setlist{nolistsep}
\begin{document}\large 
	\maketitle
	\setlength{\footskip}{20pt}
	\setlength{\oddsidemargin}{ -1cm}
	\setlength{\evensidemargin}{\oddsidemargin}
	\setlength{\textwidth}{15.50cm}
	\vspace{-.8cm}
	
	\setcounter{page}{1}
	
	\setlength{\oddsidemargin}{-.6cm}  
	\setlength{\evensidemargin}{\oddsidemargin}
	\setlength{\textwidth}{17.00cm}
		\thispagestyle{fancy}
	\fancyhf{}
	\fancyfoot[R]{\thepage}
	\fancyfoot[L]{Bowen Li, Gongxiang Liu:  School of Mathematics, Nanjing University, Nanjing, P. R. China. \\ B.Li, e-mail: DZ21210002@smail.nju.edu.cn \\ 
 G.Liu, e-mail: gxliu@nju.edu.cn.\\
 }
	\fancyhead{} 
	\renewcommand{\footrulewidth}{1pt}
	\renewcommand{\headrulewidth}{0pt}
	\par 
 \begin{abstract}
     This paper extends the foundational reflection theory of Nichols algebras to the setting of certain coquasi-Hopf algebras. Our primary motivation arises from the classification  of  pointed finite-dimensional coquasi-Hopf algebras. We develop a reflection theory for tuples of simple Yetter-Drinfeld modules in the category $\GG$, where $G$ is a finite group and $\Phi$ is a 3-cocycle on $G$. We prove that such a tuple gives rise to a semi-Cartan graph if it admits all reflections. Consequently, its Weyl groupoid is well-defined. We further establish several criteria for the finite-dimensionality of Nichols algebras in terms of the associated semi-Cartan graph. As an application, we provide a new proof for the infinite-dimensionality of a specific class of Nichols algebras previously studied in \cite{huang2024classification}, bypassing extensive computational arguments.

 \end{abstract}

\section{Introduction}
The classification of finite-dimensional pointed Hopf algebras has been one of the central problems in the theory of Hopf algebras over the past three decades. A crucial breakthrough in this direction was achieved through the systematic study of Nichols algebras. Heckenberger introduced the revolutionary concepts of semi-Cartan graphs and Weyl groupoids [\citealp{Hec06},\citealp{Hec09},\citealp{Hec10}], which provide powerful combinatorial tools for understanding the structure and classification of Nichols algebras.\par 
The reflection theory of Nichols algebras represents one of the most important developments in the field.
The fundamental theorem states that a tuple of simple Yetter-Drinfeld modules over a Hopf algebra with bijective antipode will give rise to a semi-Cartan graph provided that it admits all reflections [\citealp{AHS10},\citealp{HS10}]. This result was proved again using a categorical approach  in the setting of Hopf algebras in \cite{HSdualpair}.\par   
We recall that a semi-Cartan graph consists of the following data,
\begin{definition}
   Let $\mathbb{I}$ be a non-empty finite set,   \(\mathcal{X}\) a non-empty set, $r: \mathbb{I} \times \mathcal{X}\rightarrow 
   \mathcal{X}$, $A: \mathbb{I} \times \mathbb{I} \times \mathcal{X}\rightarrow \mathbb{Z}$ maps. For all $i,j \in \mathbb{I}$ and ${X} \in \mathcal{X}$ we write $r_i(X)=r(i,X)$, $a_{ij}^X=A(i,j,X)$ and $A^X=(a_{ij}^X)_{i,j \in \mathbb{I}} \in \mathbb{Z}^{\mathbb{I} \times \mathbb{I}}$. The quadruple \(\mathcal{G} = \mathcal{G}(\mathbb{I}, \mathcal{X}, (r_i), (A^X)_{X\in \mathcal{X}})\) is called a semi-Cartan graph if for all $X \in \mathcal{X}$, the matrix $A^X$ is a generalized Cartan matrix, and the following axioms hold.
\\  \textrm{$(CG1)$} For all \(i \in I\), \(r_i^2 = \mathrm{id}_{\mathcal{X}}\).
    \\ \textrm{$(CG2)$} For all \(i \in I\), \(X \in \mathcal{X}\), \(A^X\) and \(A^{r_i(X)}\) have the same \(i\)-th row.

\end{definition}

 The  structure underlying a semi-Cartan graph encodes the combinatorial data necessary to define the Weyl groupoid $\mathcal{W}(\mathcal{G})$, which parametrizes the action of reflections on the set of real roots. 
 Given a semi-Cartan graph $\mathcal{G}=(\mathbb{I},\mathcal{X},r,(A^X)_{X\in \mathcal{X}})$, the Weyl groupoid of $\mathcal{G}$ is defined as follows.
\begin{definition}
   Let $\mathcal{D}(\mathcal{X},\operatorname{End}(\mathbb{Z}^\mathbb{I}))$  denote the category with objects $\operatorname{Ob}\mathcal{D}(\mathcal{X},\operatorname{End}(\mathbb{Z}^\mathbb{I}))=\mathcal{X}$, and morphisms 
    $$ \operatorname{Hom}(X,Y)=\{ (Y,f,X)\mid f\in \operatorname{End}(\mathbb{Z}^\mathbb{I})\}, $$
    where the composition of morphisms is defined by 
    $$(Z,g,Y)\circ (Y,f,X)=(Z,gf,X),  \ \text{for all} \  X,Y,Z \in \mathcal{X}, \ f,g \in \operatorname{End}(\mathbb{Z}^\mathbb{I}).$$
   Let $\alpha_i$, $1\leq i\leq \theta$ be the standard basis of  $\mathbb{Z}^\mathbb{I}$, and 
    $$ s_i^X \in \operatorname{Aut}(\mathbb{Z}^{\mathbb{I}}), \ s_i^X(\alpha_j)=\alpha_j-a_{ij}^X\alpha_i, \text{for all} \ j.$$
    We call the smallest subcategory of $\mathcal{D}(\mathcal{X},\operatorname{End}(\mathbb{Z}^\mathbb{I}))$ which contains all morphisms $(r_i(X),s_i^X,X)$ with $i \in \mathbb{I}$, $X \in \mathcal{X}$ the Weyl groupoid of $\mathcal{G}$, denoted by $\mathcal{W}(\mathcal{G})$.
\end{definition}
 In the context of Nichols algebras, the Weyl groupoid captures the symmetries and reflection properties of the underlying braided vector space, establishing a bridge between the algebraic structure of Nichols algebras and the combinatorics of root systems [\citealp{AHS10}].\par 
The interplay between Weyl groupoid theory and reflection theory has led to spectacular progress in the classification of finite-dimensional Nichols algebras. The classification program, initiated by Andruskiewitsch and Schneider in a series of  papers [\citealp{AS98},\citealp{AS00},\citealp{AS02},\citealp{AS10}], aims to determine all finite-dimensional pointed Hopf algebras over algebraically closed fields of characteristic zero whose group of group-like elements is abelian. A cornerstone of this classification is Heckenberger's complete classification of finite-dimensional Nichols algebras of diagonal type [\citealp{Hec09}]. Building on this classification, Angiono made a breakthrough contribution by providing explicit presentations by generators and relations for all Nichols algebras of diagonal type [\citealp{A13},\citealp{A15}]. The lifting problem of Nichols algebras of diagonal type has been addressed through cocycle deformation [\citealp{AA+14},\citealp{AG19},\citealp{AA17}]. The classification of finite-dimensional Nichols algebras over non-abelian groups  has achieved excellent progress as well [\citealp{AGn99},\citealp{AGn03},\citealp{AHS10},\citealp{MS00},\citealp{FK99},\citealp{Gn00},\citealp{Gn+11},\citealp{HLV12},\citealp{HS10},\citealp{HS10b},\citealp{HS10c},\citealp{HS14},\citealp{Hs15},\citealp{HV17}].\par 
In this paper, we extend the reflection theory to the setting of some special coquasi-Hopf algebras. Our motivation comes from the classification of coradically graded pointed finite-dimensional coquasi-Hopf algebras, which has been studied in [\citealp{rank1},\citealp{coquasitriangular},\citealp{tame-rep},\citealp{FQQR2},\citealp{QQG},\citealp{huang2024classification},\citealp{LiLiu1}].
In the classification of pointed Hopf algebras over abelian groups, Nichols algebras of non-diagonal type do not appear. However, in the coquasi-Hopf algebra setting,  Nichols algebras of non-diagonal type do occur and present greater challenges. The authors in \cite{huang2024classification} used extensive computations to establish that a certain class of such  Nichols algebras is infinite-dimensional. We later provided an alternative proof of their main result via categorical Morita equivalence 
 \cite{LiLiu1}.  Nonetheless, these arguments are not entirely natural, and we hope that the theory of semi-Cartan graphs can be extended to coquasi-Hopf algebras. If such a generalization holds, it could be applied to the classification of finite-dimensional coquasi-Hopf algebras over non-abelian groups.\par 
 Our approach combines the methods of \cite{Simon} and \cite{HSdualpair}. The next theorem generalizes [\citealp{HSdualpair}, Theorem 7.1] to the case of coquasi-Hopf algebras with bijective antipodes under some minor conditions, which is crucial to the description of the reflection of Nichols algebras.
\begin{theorema} \textup{({Theorem} \ref{C-cor3.8})}
    Let $H$ be a coquasi-Hopf algebra with bijective antipode, and $\mathcal{C}=\HH$, (see Definition \ref{C-def2.5} for details). Suppose $N\in \HH$, and $\bB(N)$ is finite-dimensional. We have the following equivalence of braided monoidal categories 
 \begin{align*}
  \Omega:{^{\bB(N)}_{\bB(N)}\mathcal{YD}(\mathcal{C})}& \cong {^{\bB(N^*)}_{\bB(N^*)}\mathcal{YD}(\mathcal{C})}.
 \end{align*}
\end{theorema}
It is worth mentioning that in our cases, the above two categories are not strict. This makes the definition of the functor $\Omega$ more complicated and adds some challenges to studying its properties.\par 
Now we restrict to the case of pointed cosemisimple coquasi-Hopf algebras.
Let $G$ be a finite group, and $\Phi$ be a $3$-cocycle on $G$, which gives rise to the Yetter-Drinfeld module category $\GG$ \cite{QQG}.  Let $\theta \geq 1$ be a positive integer, $\mathbb{I}=\{1,2,..,\theta\}$ and $M = (M_1, \ldots, M_{\theta})$, where $M_1, \ldots, M_{\theta} \in \GG$ are  irreducible Yetter-Drinfeld modules. We always denote $\bB(M)$ as the Nichols algebra of $M_1\oplus \cdots \oplus M_{\theta}$ in the following context.  Suppose $M$ admits the $i$-th reflection for
some $1\leq i \leq \theta$ (see Definition \ref{C-def5.1}). Then the $i$-th reflection of $M$ is denoted by $R_i(M)=(V_1,\ldots,V_\theta)$, which is a tuple consisting of finite-dimensional irreducible Yetter-Drinfeld modules. There is a natural explanation  of reflection of Nichols algebra in terms of $\Omega$:
\begin{theoremb}\textup{(Theorem \ref{C-thm 3.15})}
   Under the above assumptions on M, suppose $M$ admits the $i$-th reflection for $1\leq i \leq \theta$, and $\bB(M_i)$ is finite-dimensional. Then there is an isomorphism of Hopf algebras in $\GG$:
     \begin{equation}
         \Theta_i:\bB(R_i(M))\cong \Omega_i\left(\bB(M)^{\operatorname{co}\bB(M_i)}\right)\# \bB(M_i^*).
     \end{equation}
\end{theoremb}
Readers who are familiar with reflection theory in the Hopf algebra context will recognize this formula, we emphasize that our contribution lies in establishing this result in  the framework of coquasi-Hopf algebras. The extension from Hopf algebras to coquasi-Hopf algebras introduces  technical challenges, particularly in Sections 3 and 4 where the non-strict monoidal structure necessitates a  reworking of the reflection theory.with the help of the above theorems, our main result follows naturally.
\begin{theoremc}\textup{(Theorem \ref{C-thm5.9})}
Suppose $M$ admits all reflections and 
    let $\mathcal{X}=\{ [P]\mid P \in \mathcal{F}_{\theta}(M) \}$, see Definition \ref{C-def5.2} for related notations. We denote the generalized Cartan matrix $A^{[X]}=(a_{ij}^X)_{i,j \in \mathbb{I}}$ for each $[P] \in \mathcal{X}$, where $(a_{ij}^X)_{i,j \in \mathbb{I}}$ is defined in Lemma \ref{C-lem5.4}. Let $r$ be the following map, $$r: \mathbb{I} \times \mathcal{X} \rightarrow \mathcal{X},\ i \times X \mapsto [R_i(X)].$$ Furthermore, we assume $\bB(P_i)$ is finite-dimensional for each $P \in \mathcal{X}$ and $1\leq i \leq \theta$, then 
    $$ \mathcal{G}(M)=(\mathbb{I},\mathcal{X},r,(A^X)_{X\in \mathcal{X}}),$$
   is a semi-Cartan graph.
\end{theoremc}
After establishing that $M$ gives rise to a semi-Cartan graph, several additional criteria for determining the finite-dimensionality of the Nichols algebra can also be derived. We show that if $\mathcal{B}(M)$ is finite-dimensional,  then the semi-Cartan graph $\mathcal{G}(M)$ is a finite semi-Cartan graph. Furthermore, if $\mathcal{G}(M)$ is         a standard finite semi-Cartan graph, that is $A^X=A^Y$, for all $X,Y \in \mathcal{X}$, then $A^M $ must be a finite Cartan matrix. 
As an application of these results, we provide a third proof of [\citealp{huang2024classification}, Proposition 4.1].\par 
The paper is organized as follows. Section 2 is dedicated to preliminaries. In Section 3, we establish a braided monoidal equivalence between categories of Yetter-Drinfeld modules. Section 4 deals with projections of Nichols algebras; the results obtained here are essential to the study of reflections of Nichols algebras. In Section 5, we prove that, under minor assumptions, a tuple of simple Yetter-Drinfeld modules gives rise to a semi-Cartan graph. Finally, in Section 6, we apply the theory of semi-Cartan graph to the classification of finite-dimensional Nichols algebras over certain coquasi-Hopf algebra.

 \section{Preliminaries}
Throughout this paper, $\mathbbm{k}$ is an algebraically closed field of characteristic zero and all linear spaces are over $\mathbbm{k}$.

\subsection{Coquasi-Hopf algebras}
By definition, coquasi-Hopf algebras are exactly the dual  of Drinfeld's quasi-Hopf algebras \cite{Drinfeld}. Their formal definition can be given as follows.

\begin{definition}
    A coquasi-Hopf algebra is a coalgebra $(H, \Delta, \varepsilon)$ equipped with a compatible quasi-algebra structure and an antipode $(\mathcal{S},\alpha, \beta)$. Namely, there exist:
\begin{itemize}
    \item Two coalgebra homomorphisms:
    \[
    m : H \otimes H \rightarrow H, \quad a \otimes b \mapsto ab,
    \]
    \[
    \mu : \mathbbm{k} \rightarrow H, \quad \lambda \mapsto \lambda 1_{H},
    \]
    
    \item A convolution-invertible map $\Phi : H^{\otimes 3} \rightarrow \mathbbm{k}$ called an \textit{associator},
    
    \item A coalgebra antimorphism $\mathcal{S} : H \rightarrow H$,
    
    \item Two linear functions $\alpha, \beta : H \rightarrow \mathbbm{k}$
\end{itemize}
such that for all $a, b, c, d \in H$ the following equalities hold:
\begin{align}
    a_1(b_1c_1)\Phi(a_2, b_2, c_2) &= \Phi(a_1, b_1, c_1)(a_2b_2)c_2, \\
    1_{H}a &= a = a1_{H}, \\
    \Phi(a_1, b_1, c_1d_1)\Phi(a_2b_2, c_2, d_2) &= \Phi(b_1, c_1, d_1)\Phi(a_1, b_2c_2, d_2)\Phi(a_2, b_3, c_3), \\
    \Phi(a, 1_{H}, b) &= \varepsilon(a)\varepsilon(b), \\
    \mathcal{S}(a_1)\alpha(a_2)a_3 &= \alpha(a)1_{H}, \quad a_1\beta(a_2)\mathcal{S}(a_3) = \beta(a)1_{H}, \\
    \Phi(a_1, \mathcal{S}(a_3), a_5)\beta(a_2)\alpha(a_4) &= \Phi^{-1}(\mathcal{S}(a_1), a_3, \mathcal{S}(a_5))\alpha(a_2)\beta(a_4) = \varepsilon(a).
\end{align}
\end{definition}

Throughout this paper, we use the Sweedler sigma notation $\Delta(a) = a_1 \otimes a_2$ for the coproduct and $a_1 \otimes a_2 \otimes \cdots \otimes a_{n+1}$ for the result of the $n$-iterated application of $\Delta$ on $a$. We say $H$ has a bijective antipode if $\mathcal{S}$ is bijective.

A coquasi-Hopf algebra $H$ is pointed if the underlying coalgebra is so. Given a  coquasi-Hopf algebra, let $\{ H_n \}_{n \geq 0}$ be its coradical filtration, and $$ \operatorname{gr}(H):=H_0 \oplus H_1/H_0 \oplus  H_2/H_1 \oplus \cdots $$ the corresponding coradically graded coalgebra. Then naturally $(\operatorname{gr}(H),\operatorname{gr}(\Phi))$ inherits from $H$  a graded coquasi-Hopf algebra structure. Furthermore, for homogeneous elements $a,b,c \in \operatorname{gr}(H)$, $\operatorname{gr}(\Phi)(a,b,c)= 0$ unless they all lie in $H_0$.\par
Following the definition in [\citealp{preantipode}], we now recall the definition of a preantipode, plays an important role in subsequent calculations.
\begin{definition}\textup{[\citealp{preantipode}, Definition 3.6]}
    Let $(H,\Phi)$ be a coquasi-bialgebra, a preantipode for $H$ is a $\mathbbm{k}$-linear map $\mathbb{S}:H \rightarrow H$ such that, for all $h \in H$,
    \begin{align*}
        \mathbb{S}(h_1)_1h_2\otimes \mathbb{S}(h_1)_2&=1_H \otimes \mathbb{S}(h),\\
        \mathbb{S}(h_2)_1\otimes h_1\mathbb{S}(h_2)_2&=\mathbb{S}(h)\otimes 1_H,\\
\Phi(h_1,\mathbb{S}(h_2),h_3)&=\varepsilon(h).
    \end{align*}
\end{definition}
Furthermore, [\citealp{preantipode}, Remark 3.7] provides us  a useful property:
\begin{equation}
    h_1\mathbb{S}(h_2)=\varepsilon(\mathbb{S}(h))1_H=\mathbb{S}(h_1)h_2, \ \text{for all} \ h \in H.
\end{equation}
Meanwhile, when $H$ is equipped with the structure of a coquasi-Hopf algebra with an antipode $(\mathcal{S}, \alpha, \beta)$, then preantipode for $H$ always exists.
\begin{lemma}\textup{[\citealp{preantipode}, Theorem 3.10]}\label{C-lem1.3}
Let $H$ be a coquasi-Hopf algebra with an  antipode $(\mathcal{S}, \alpha, \beta)$, then 
\begin{equation}
    \mathbb{S}:= \beta * \mathcal{S} *\alpha 
\end{equation}
is a preantipode for $H$, where $*$ denotes convolution product. 
\end{lemma}
\begin{exm}\rm
    A baby example of coquasi-Hopf algebra is  $(\mathbbm{k}G,\Phi)$, where $G$ is a group and $\Phi$ a normalized 3-cocycle on $G$. It is well known that the group algebra $\mathbbm{k}G$ is a Hopf algebra with:
\[
\Delta(g) = g \otimes g, \quad \mathcal{S}(g) = g^{-1}, \quad \varepsilon(g) = 1 \quad \text{for any } g \in G.
\]
By extending $\Phi$ linearly, $\Phi : (\mathbbm{k}G)^{\otimes 3} \rightarrow \mathbbm{k}$ becomes a convolution-invertible map. Define two linear functions $\alpha, \beta : \mathbbm{k}G \rightarrow \mathbbm{k}$ by:
\[
\alpha(g) := \varepsilon(g), \quad \beta(g) := \frac{1}{\Phi(g, g^{-1}, g)}
\]
for any $g \in G$. Then $\mathbbm{k}G$ together with these $\Phi, \alpha$ and $\beta$ becomes a coquasi-Hopf algebra. We denote this resulting coquasi-Hopf algebra by $(\mathbbm{k}G, \Phi)$.\label{C-exm1.2} For all $g\in G$, we define 
    $$\mathbb{S}(g)=\frac{1}{\Phi(g,g^{-1},g)}g^{-1}. $$
    Then Lemma \ref{C-lem1.3} shows $\mathbb{S}$ is a preantipode for $(\mathbbm{k}G,\Phi)$.
\end{exm}

 \subsection{ Yetter-Drinfeld module categories over coquasi-Hopf algebras}
 We now turn our attention to the Yetter-Drinfeld module category structure.
 The definition of the Yetter-Drinfeld module category $\HH$ over an arbitrary coquasi Hopf-algebra $H$ was already given in [\citealp{YD-module}]. It is defined as the center $\mathcal{Z}({{^H}\mathscr{M}})$ of the comodule category ${^H\mathscr{M}}$. Hence it is braided monoidal equivalent to $\textbf{Rep}(D(H))$. Since it plays a crucial role in our paper, we recall this definition below.
 \begin{definition}\textup{[\citealp{YD-module}, Definition 3.1]}\label{C-def2.5}
Let $H$ be a coquasi-Hopf algebra with associator $\Phi$. A left-left Yetter--Drinfeld module over $H$ is a triple $(V, \delta_V, \rhd)$ such that:

\begin{itemize}
    \item $(V, \delta_V)$ is a left comodule of $H$ and we denote $\delta_V(v)$ by $v_{-1} \otimes v_0$ as usual;
    
    \item $\rhd: H \otimes V \to V$ is a $\mathbbm{k}$-linear map satisfying for all $h, l \in H$ and $v \in V$:
    \begin{align}
   & (hl) \rhd v = \frac{\Phi(h_2, (l_2 \rhd v_0)_{-1}, l_3)}{\Phi(h_1, l_1, v_{-1})\Phi((h_3 \rhd (l_2 \rhd v_0)_0)_{-1}, h_4, l_4)} 
    (h_3 \rhd (l_2 \rhd v_0)_0)_0, \\
  &  1_H \rhd v = v, \\ 
  & (h_1 \rhd v)_{-1} h_2 \otimes (h_1 \rhd v)_0 = h_1 v_{-1} \otimes (h_2 \rhd v_0). \label{C-1.9}
    \end{align}
\end{itemize}
A morphism $f:(V,\delta_V,\rhd)\rightarrow (V',\delta_{V'},\rhd')$ is a colinear map $f:(V,\delta_V)\rightarrow (V',\delta_{V'})$ such that $f(h \rhd v)=h \rhd' f(v)$ for all $h\in H$.
\end{definition}
 The category $\HH$ is a $\mathbbm{k}$-linear  braided monoidal  abelian category over the field $\mathbbm{k}$. The unit object of $\HH$ is $\mathbbm{k}$, which is regarded as an object in $\HH$ via trivial structures. For $V,W \in \HH$, the tensor product of Yetter-Drinfeld modules is defined by:
\[
(V, \delta_V, \rhd) \otimes (W, \delta_W, \rhd) = (V \otimes W, \delta_{V \otimes W}, \rhd),
\]
where  $\delta_{V \otimes W}$ is given by 
\[
\delta_{V \otimes W} (v \otimes w) = v_{-1} w_{-1} \otimes v_0 \otimes w_0,
\]
and
\begin{equation}\label{C-1.12}
h \rhd (v \otimes w) =
\frac{\Phi (h_1, v_{-1} , w_{-2})\Phi((h_2\rhd v_0)_{-1}, (h_4 \rhd w_0 )_{-1}, h_5)}{\Phi ((h_2 \rhd v_0)_{-2}, h_3, w_{-1}) }
  (h_2 \rhd v_0)_0 \otimes (h_4 \rhd w_0)_0.
\end{equation}
The braiding $
c_{V,W} : V \otimes W \rightarrow W \otimes V
$
is given by:
\begin{equation}
c_{V,W} (v \otimes w) = (v_{-1} \rhd w) \otimes v_0.
\end{equation}\par 
It is worth noting that if $H$ has a bijective antipode, then equation (\ref{C-1.9}) has an equivalent form, which will prove useful in later calculations.
In order to simplify notation, we recall the following linear maps 
\( p \), \( q \), \( s \), \( t \)\(\in \operatorname{Hom}(H\otimes H, \mathbbm{k}) \), which is introduced in \cite{rightyd}.
\[
\begin{aligned}
p(h, g) &= \Phi^{-1}(h, g_1 \beta(g_2), \mathcal{S}(g_3)), \\
q(h, g) &= \Phi(h, g_3, \alpha (\mathcal{S}^{-1}(g_2)) \mathcal{S}^{-1}(g_1)), \\
s(h,g)&=\Phi(\mathcal{S}^{-1}(h_3)\beta(\mathcal{S}^{-1}(h_2)),h_1,g),\\
t(h,g)&=\Phi^{-1}(\mathcal{S}(h_1)\alpha(h_2),h_3,g).
\end{aligned}
\]
for any \( h, g \in H \).
These maps satisfy the following properties:
\begin{align}\label{C-1.14}
(h_1 g_1) \mathcal{S}(g_3) p(h_2, g_2) &= h_2 p(h_1, g), \\
(h_2 g_3) \mathcal{S}^{-1}(g_1) q(h_1, g_2) &= h_1 q(h_2, g),\label{C-1.15}\\ \label{C-1.16}
q(h_1 g_1, \mathcal{S}(g_3)) p(h_2, g_2) &= \varepsilon(h) \varepsilon(g), \\
\mathcal{S}(h_1)(h_3g_2)t(h_2,g_1)&=g_1t(h,g_2),\label{C-2.17}\\
s(\mathcal{S}(h_1),h_3g_2)t(h_2,g_1)&=\varepsilon(h)\varepsilon(g).\label{C-2.18}
\end{align}
With the above notations established, we can now prove the following lemma.
\begin{lemma} \label{C-lem2.6}
Suppose $H$ has a bijective antipode.\par 
 \textup{(1)} The equation (\ref{C-1.9}) is equivalent to  
\begin{equation}\label{C-1.17}
    (h\rhd v)_{-1} \otimes (h\rhd v)_0=p((h_3 \rhd v_0)_{-1},h_4)   q(h_1v_{-2},\mathcal{S}(h_6))    (h_2v_{-1})\mathcal{S}(h_5)      \otimes (h_3 \rhd v_0)_0.
\end{equation}\par 
\textup{(2)} The equation (\ref{C-1.9}) implies:
\begin{equation}
    v_{-1} \otimes (h\rhd v_0)=t(h_3,v_{-1})s(\mathcal{S}(h_1),(h_4\rhd v_0)_{-1}h_6)\mathcal{S}(h_2)((h_4\rhd v_0)_{-2}h_5)\otimes (h_4 \rhd v_0)_0.
\end{equation}
\end{lemma}

\begin{proof}
(1): Suppose $V \in \HH$ with a linear map $\rhd : H \otimes V \rightarrow V$ satisfying (\ref{C-1.9}). For any $v \in V, h\in H$, then we have
\begin{align*}
&(h\rhd v)_{-1} \otimes (h\rhd v)_0=  \varepsilon(h_2)\varepsilon((h_1 \rhd v)_{-2}) (h_1 \rhd v)_{-1}\otimes (h_1 \rhd v)_{0}  \\
&\overset{(\ref{C-1.15})}{=}q((h_1\rhd v)_{-3}h_2,S(h_4))p((h_1 \rhd v)_{-2},h_3)(h_1 \rhd v)_{-1}\otimes (h_1 \rhd v)_{0}\\
&\overset{(\ref{C-1.14})}{=}p((h_1 \rhd v)_{-1},h_4)q((h_1\rhd v)_{-3}h_2,\mathcal{S}(h_6))((h_1 \rhd v)_{-2}h_3)\mathcal{S}(h_5)\otimes (h_1 \rhd v)_{0}\\
& \overset{(\ref{C-1.9})}{=}p((h_3 \rhd v_0)_{-1},h_4)   q(h_1v_{-2},\mathcal{S}(h_6))    (h_2v_{-1})\mathcal{S}(h_5)      \otimes (h_3 \rhd v_0)_0.
\end{align*}
Conversely, if (\ref{C-1.17}) holds, we have
\begin{equation}\label{C-1.18}
     \varepsilon((h\rhd v)_{-1})(h\rhd v)_0=p((h_2 \rhd v_0)_{-1},h_3)q(h_1v_{-1},\mathcal{S}(h_4))(h_2\rhd v_0)_0.
\end{equation}
The origin equation comes from the following calculation, \begin{align*}
    (h_1\rhd v)_{-1}h_2 \otimes (h_1 \rhd v)_0&=p((h_3 \rhd v_0)_{-1},h_4)   q(h_1v_{-2},\mathcal{S}(h_6))    ((h_2v_{-1})\mathcal{S}(h_5))h_7 \otimes (h_3 \rhd v_0)_0\\
    &\overset{(\ref{C-1.15})}{=}p((h_3 \rhd v_0)_{-1},h_4)   q(h_2v_{-1},\mathcal{S}(h_5)) h_1v_{-2}\otimes (h_3\rhd v_0)_0 \\
    &\overset{(\ref{C-1.18})}{=}h_1v_{-1}\otimes(h_2 \rhd v_0).
\end{align*}\par 
(2) We have 
\begin{align*}
    v_{-1}\otimes ( h\rhd v_0)&=\varepsilon(v_{-1})\varepsilon(h_1)v_{-2}\otimes (h_2\rhd v_0)\\
    &\overset{(\ref{C-2.18})}{=}s(\mathcal{S}(h_1),h_3v_{-1})t(h_2,v_{-2})v_{-3}\otimes (h_4\rhd v_0)\\
    &\overset{(\ref{C-2.17})}{=}s(\mathcal{S}(h_1),h_5v_{-1})t(h_3,v_{-3})\mathcal{S}(h_2)(h_4v_{-2})\otimes h_6\rhd v_0\\
    &\overset{(\ref{C-1.9})}{=}s(\mathcal{S}(h_1),(h_4\rhd v_0)_{-1}h_6)t(h_3,v_{-1})\mathcal{S}(h_2)((h_4\rhd v_0)_{-2}h_5)\otimes{ (h_4\rhd v_0)_0}.
\end{align*}
\end{proof}

\begin{exm}\rm

For our purpose, we aim to describe  Yetter-Drinfeld modules over coquasi-Hopf algebra of the form $(\mathbbm{k}G, \Phi)$, where  $G$ is a finite  group and $\Phi$ a $3$-cocycle on $G$. Assume that $V$ is a left $kG$-comodule with a comodule structure map $\delta_L : V \rightarrow kG \otimes V$. Define

\[
{}^gV := \{ v \in V \mid \delta_L(v) = g \otimes v \}.
\]
Thus
\[
V = \bigoplus_{g \in G} {}^gV.
\]

Here we call $g$ the degree of the elements in ${}^gV$ and denote $\deg v = g$ for $v \in {}^gV$.

The left $\mathbbm{k}G$-comodule $(V, \delta_L)$ is a left-left Yetter--Drinfeld module over the coquasi-Hopf algebra $H = (\mathbbm{k}G, \Phi)$ if there is a linear map $\rhd : G \otimes V \rightarrow V$ such that for all $e, f \in G$ and $v \in {}^gV$:
\begin{align}
&e \rhd (f \rhd v) = \frac{\Phi(e, f, g)\Phi(efgf^{-1}e^{-1}, e, f)}{\Phi(e, fgf^{-1}, f)} (ef) \rhd v, \\
&1_H \rhd v = v, \\
&e \rhd v \in {}^{ege^{-1}}V.
\end{align}

The category of all left-left Yetter--Drinfeld modules over $(\mathbbm{k}G, \Phi)$ is denoted by $\GG$. Similarly, one can define left-right, right-left and right-right Yetter--Drinfeld modules over $(\mathbbm{k}G, \Phi)$. 

It is well-known
 that $\GG$ is a braided tensor category. More precisely, for any $M, N \in \GG$, the structure maps of $M \otimes N$ as a left-left Yetter--Drinfeld module are given by:
\begin{align}
\delta_L(m_g \otimes n_h) &:= gh \otimes m_g \otimes n_h, \\
x \rhd (m_g \otimes n_h) &:= \frac{\Phi(x, g, h)\Phi(xgx^{-1}, xhx^{-1}, x)}{\Phi(xgx^{-1}, x, h)} x \rhd m_g \otimes x \rhd n_h,
\end{align}
for all $x, g, h \in G$ and $m_g \in {}^gM$, $n_h \in {}^hN$.

The associativity constraint $a$ and the braiding $c$ of $\GG$ are given respectively by:
\begin{align}
a((u_e \otimes v_f) \otimes w_g) &= \Phi(e, f, g)^{-1}u_e \otimes (v_f \otimes w_g), \\
c(u_e \otimes v_f) &= e \rhd v_f \otimes u_e,
\end{align}
for all $e, f, g \in G$, $u_e \in {}^eU$, $v_f \in {}^fV$, $w_g \in {}^gW$ and $U, V, W \in \GG$.
\end{exm}
In general, let $\mathcal{C}$ be a braided monoidal category and $A$ a Hopf algebra in $\mathcal{C}$. The category of Yetter-Drinfeld modules  $\AAC$ is well-defined. The reader may refer to [\citealp{rootsys}, Section 3] for details.

\subsection{Bosonization for  coquasi-Hopf algebras}
Bosonization for Hopf algebras is well-known. In [\citealp{quasibon}], the authors introduced bosonization for quasi-Hopf algebras. Dually, bosonization for coquasi-Hopf algebras is given in \cite{YD-module}.
Let $H$ be a coquasi-Hopf algebra with associator $\Phi$. Suppose $R$  is a Hopf algebra in $\HH$, we denote 
$$ r^1 \otimes r^2:= \Delta_R(r).$$
\begin{lemma}\textup{[\citealp{YD-module}, Theorem 5.2]}
    Let us consider on $N:=R\otimes H$ the following structures:
\begin{align*}
(r \otimes h)  (s \otimes k) &= \frac{\Phi(h_2,s_{-1}, k_2)\Phi(r_{-1},(h_3 \rhd s_0)_{-1},h_5k_4)}{\Phi (r_{-2},h_1, s_{-2}k_1)\Phi((h_3 \rhd s_0)_{-2}, h_4, k_3)}   r_0  (h_3 \rhd s_0)_0 \otimes h_6k_5, \\
 u_B(k) &= k1_R \otimes 1_M,\\
\Delta_B(r \otimes h) &= \Phi^{-1}(r_{-1}^1, r_{-2}^2, h_1)r_0^1 \otimes r_{-1}^2 h_2 \otimes r_0^2 \otimes h_3, \\
\varepsilon_B(r \otimes h) &= \varepsilon_R(r)\varepsilon_M(h),\\
\omega_B((r \otimes h), (s \otimes k), (t \otimes l))& = \varepsilon_R(r)\varepsilon_R(s)\varepsilon_R(t)\Phi(h, k, l),
\end{align*}
where $r,s,t \in R$, $h,k,l \in H$.
With above operations, $N$ is a coquasi-Hopf algebra.
\end{lemma}

Now suppose $N$ and $H$ are both coquasi-Hopf algebras.  Furthermore, assume there exist morphisms of coquasi-Hopf algebras
$$ \pi: N \rightarrow H \ \text{ and } \sigma: H \rightarrow N$$
such that $\pi\sigma=\operatorname{id}_N$.
\begin{lemma}\textup{[\citealp{YD-module}, Theorem 5.8]} \label{C-Lemma 1.6}
   Under above assumptions, $L:=N^{\operatorname{co}H}$ is a Hopf algebra in $\HH$.
     For all $a \in N$, we denote 
   \begin{equation}
    \tau(a)=\Phi_N(a_1,\sigma\mathbb{S}\pi(a_3)_1, a_4)a_2\sigma \mathbb{S}\pi(a_3)_2.
   \end{equation} 
   Then $\tau: N \rightarrow L$ is a well-defined map,
   where, for all $r, s \in L$, $h\in H$, $k \in \mathbbm{k}$, the Yetter-Drinfeld module structure is given by
   \begin{align*}
       \operatorname{ad}(h)(r)&:=\tau(\sigma(h)r)=\Phi_{H}(h_{1}r_{-1},\mathbb{S}(h_3)_1,h_4)(h_2r_0)\mathbb{S}(h_3)_2, \\  r_{-1}\otimes r_0 &:=\rho_R(r):=\pi(r_1)\otimes r_2.
       \end{align*} Moreover, $L$ is a Hopf algebra in $\HH$ via
       \begin{align*}
           &m_L(r\otimes s):=rs, \ \ \ \  \ \  \ u_L(k)=k1_N \\ 
       &\Delta_L(r):=\tau(r_1) \otimes \tau(r_2), \ \ \ \ \varepsilon_L(r):=\varepsilon_N(r).
   \end{align*}
   Furthermore, there is an isomorphism of coquasi-Hopf algebra $\phi: L\# H \rightarrow N$ given by 
   $$ \phi(r \otimes h)=r \sigma(h), \ \  \phi^{-1}(a)=\tau(a_1)\otimes \pi(a_2).$$
\end{lemma}
For the reader's convenience, we provide an example about $\GG$,
where $G$ is a finite group and $\Phi$ is a $3$-cocycle on $G$. Suppose $R$ is a Hopf algebra in $\GG$. Since $R$ is a left $G$-comodule, there is a
$G$-grading on $R$: $$R =\bigoplus_{x \in G} {^xR},$$
where $^xR:=\{ X\in R\mid \rho(x)=x\otimes X \}$.  For simplicity, in this paper, homogeneous elements in $\GG$ are denoted by capital letters, say $X,Y,Z...$, and the associated degrees are denoted by their lowercase letters, say $x,y,z$...
\begin{exm}\rm \label{C-ex 1.7}
    Let us maintain the assumptions on $R$ from above. For all $X,Y,Z \in R$, $h,g,k \in G$, we consider the following structure on $M=R \otimes  \mathbbm{k}G$.
    \begin{align*}
         &(X \otimes  h)(Y\otimes g)=\frac{\Phi(h,y,k)\Phi(x,hyh^{-1},hg)}{\Phi(x,h,yg)\Phi(hyh^{-1},h,g)}X(h \rhd Y)\otimes hg, \ \ \ 
         u_{H}(k)=k1_R\otimes 1_{\mathbbm{k}G}, \\
         &\Delta(X\otimes h)= \Phi^{-1}(x^1,x^2,h)X^1 \otimes x^2h\otimes X^2 \otimes h, \ \ \ \ \ \ \
         \varepsilon(X\otimes h)=\varepsilon_R(X)\varepsilon_{\mathbbm{k}G}(h),\\[2pt]
         &\Phi_M(X\otimes h, Y\otimes g, Z \otimes k)=\varepsilon_R(X)\varepsilon_R(Y) \varepsilon_R(Z)\Phi(h,g,k),\\[2pt]
        &\mathcal{S}(X\otimes g)=\Phi(g^{-1}x^{-1},x,g)(1\otimes g^{-1}x^{-1})(\mathcal{S}_R(X)\otimes 1),\\
    &\alpha(1\otimes g)=1, \ \ \alpha(X\otimes g)=0, \ \ \ \ \beta(1\otimes g)=\frac{1}{\Phi(g,g^{-1},g)}, \ \ \beta(X\otimes g)=0.
    \end{align*}
    With the above operations, $M$ is indeed a coquasi-Hopf algebra.\par 
Conversely, let us now consider the reverse construction. 
    Suppose  $M$ is  a coradically graded coquasi-Hopf algebra with coradical $(\mathbbm{k}G,\Phi)$. Then $R:=M^{\operatorname{co}\kG}$ is an object in $\GG$  equipped with the following structures:
    \begin{align*}
       &\operatorname{ad}(g)(X)=\frac{\Phi(gh,g^{-1},g)}{\Phi(g,g^{-1},g)}(gX)g^{-1}, \ \   \\ &\rho_R(X)=\pi(X_1)\otimes X_2, \ X \in {^hR}, g \in G.
       \end{align*} 
   Furthermore, $R$ carries a Hopf algebra structure in $\GG$, 
    \begin{align*}
       &m_R(X \otimes Y)=m_M(X\otimes Y), \ \ \ \ u_R(k)=\mathbbm{k}1_M,\\
       & \Delta_R(X)=\frac{1}{\Phi(x_1,x_2,x_2^{-1})}X_1 \cdot  x_2^{-1} \otimes X_2,  \ \  \ \varepsilon_R(X)=\varepsilon_M(X), \\
       & \mathcal{S}_R(X)=\frac{1}{\Phi(x,x^{-1},x)}x\cdot \mathcal{S}(X),
    \end{align*}
    where $X,Y \in R$.
\end{exm}
However, this construction alone is not sufficient for our purposes. The following condition occurs frequently in this paper. Given $R,R'$ be two Hopf algebras in $\GG$ together with a surjective Hopf algebra map $\pi': R \rightarrow R'$ and an injective map $i': R' \rightarrow R$ such that $ \pi \circ i= \operatorname{id}_R'$. Let $M=R\# \kG$, and $H=R' \# \kG$. Obviously, these maps naturally extend to yield a surjective coquasi-Hopf algebra map 
$\pi:M \rightarrow H$ and an injective coquasi-Hopf algebra map
$i:H \rightarrow M$, such that $\pi \circ i=\operatorname{id}$. \par

\begin{exm} \rm
    Let $L:=M^{\operatorname{co}H}$, then $L$ is an object in  $\HH$, with the following structures:
    \begin{align*}
    &   \operatorname{ad}: H \otimes L \rightarrow L,  h \otimes l \mapsto \Phi_{H}(h_{1}l_{-1},\mathbb{S}_H(h_3)_1,h_4)(h_2l_0)\mathbb{S}_H(h_3)_2, \\
    &\operatorname{\rho}: L \rightarrow H \otimes L, l \mapsto \pi(l_1) \otimes l_2.
    \end{align*}
    Our primary focus now shifts to analyzing the linear map $\operatorname{ad}$. Given  $g\in G$  and $X \in {^xR}$,
    $$ \operatorname{ad}(g)(X)=\frac{\Phi(gx,g^{-1},g)}{\Phi(g,g^{-1},g)}(gX)g^{-1}.$$
    Now suppose $Y$ is a $(y,1)$-primitive element in $H$, we first need to compute what $\mathbb{S}(Y)$ is,

    \begin{align*}
        \mathbb{S}(Y)= \beta * \mathcal{S} * \alpha(Y)=\frac{1}{\Phi(y,y^{-1},y)}\mathcal{S}(Y)&=\frac{1}{\Phi(y,y^{-1},y)}(1\# y^{-1})(-Y\#1)\\&=-\frac{1}{\Phi(y,y^{-1},y)}y^{-1}Y.
    \end{align*}  Then we can proceed to calculate $\operatorname{ad}(Y)(X)$,
    \begin{equation}
\begin{aligned}\label{C-1.27}
    \operatorname{ad}(Y)(X)=YX-\frac{1}{\Phi(y,y^{-1},y)}(yX)(y^{-1}Y)&=YX-\frac{\Phi(yx,y^{-1},y)}{\Phi(y,y^{-1},y)}((yX)y^{-1})Y\\&=YX-\operatorname{ad}(y)(X)Y.
\end{aligned}
\end{equation}
\end{exm}
The definition of the bosonization  naturally extends to the case that $R$  is a Hopf algebra in the
braided monoidal category $\AAC$, where $\mathcal{C}$ is now an arbitrary braided monoidal  category.
 We collect several key results from \cite{B95} and \cite{AF00} that will be essential for the construction presented in Section 3.
\begin{lemma}\label{C-lem2.12}
\textup{[\citealp{AF00}, Theorem 3.2]} Let \( H \) and \( A \) be Hopf algebras in a braided monoidal category \( \mathcal{C} \). Let \(\pi : H \to A\) and \(\iota : A \to H\) be Hopf algebra morphisms such that \(\pi \circ \iota = \mathrm{id}_{A}\). If \( \mathcal{C} \) has equalizers and \( A \otimes (-)\) preserves equalizers, there is a Hopf algebra \( R \) in the braided monoidal category $\AAC$, such that  
\[ H \cong R \# A. \]    
\end{lemma}
\begin{lemma}\textup{[\citealp{B95}, Proposition 4.2.3]}\label{C-thm1.12}
    Let \( A \) be a Hopf algebra in \( \mathcal{C} \) and \( K \) a Hopf algebra in \( \AAC \). There is an obvious isomorphism of braided monoidal categories
    \begin{equation}
        {^{K\#A}_{K\#A}\mathcal{YD}}(\mathcal{C})\cong  {^{K}_{K}\mathcal{YD}}(\AAC).
    \end{equation}
\end{lemma}

\subsection{ Nichols algebras in a braided monoidal category}
Having established the necessary framework for bosonization, we now turn to the theory of Nichols algebras. Roughly speaking, Nichols algebras serve as the analogues of ordinary symmetric algebras in more general braided monoidal categories. This concept can be defined through various equivalent approaches. 
Now let $\mathcal{C}$ be  an abelian braided monoidal category with braiding isomorphism $c$.
We present the definition of Nichols algebras in $\mathcal{C}$. \par 
Let $V$ be a nonzero object in $\mathcal{C}$; by $T(V)$ we denote the tensor algebra in $\mathcal{C}$ generated freely by $V$. Here we denote that
$$ V^{\otimes n }:=\left( \cdots\left( \left( V \otimes V\right) \otimes V\right) \cdots \otimes V\right). $$
Then
$T(V)$ is isomorphic to $\bigoplus_{n \geq 0} V^{\otimes n}$ as an object. It is standard that $T(V)$ is  a graded Hopf algebra in $\mathcal{C}$.\par 
\begin{definition}
Let $V \in \mathcal{C}$ and $I(V)$ be the largest coideal of ${T}(V)$ contained in $\bigoplus_{n \geq 2}T^n(V)$. The Nichols algebra of $V$ is defined by 
    $$ \mathcal{B}(V):=T(V)/I(V).$$
\end{definition}
For further analysis, we shall assume that  $\mathcal{C}$ is a $\mathbbm{k}$-linear braided monoidal abelian category over the field $\mathbbm{k}$. This additional structure allows us to adopt an equivalent definition which is more convenient for our purposes.
\begin{definition}\textup{[\citealp{defofnichols}, Definition 2.4]}
    For a given object $V \in \mathcal{C}$ the Nichols
algebra $\mathcal{B}(V)$ is the unique Hopf algebra in $\mathcal{C}$ that satisfies the following conditions:\par 
\text{$(1)$} The Hopf algebra $\mathcal{B}(V)$ is graded by the non-negative integers.\par
\text{$(2)$} The zeroth component of the grading satisfies $\mathcal{B}(V)_0=\mathbbm{k}$.\par
\text{$(3)$} The first component of the grading satisfies $\mathcal{B}(V)_1=V$ , and $\mathcal{B}(V)$ is generated
by $V$ as an algebra in $\mathcal{C}$.\par
\text{$(4)$} The subobject of primitive elements of $\mathcal{B}(V)$ is $V$.\par
\end{definition}
Suppose $\mathcal{C}$ and $\mathcal{D}$ are  braided monoidal abelian categories,  and $F: \mathcal{C}
\rightarrow \mathcal{D}$ is a braided monoidal equivalence of abelian category.
The following lemma plays a crucial role in our theory. 
\begin{lemma}\label{C-lem2.18}
   Let $V \in \mathcal{C}$ be an object and $\mathcal{B}(V)$ be its corresponding Nichols algebra in $\mathcal{C}$. Then  
   $$F(\mathcal{B}(V))\cong \mathcal{B}(F(V))$$ as Hopf algebras in $\mathcal{D}$.
\end{lemma}
\begin{proof}
Let $V \in \mathcal{C}$  and $J_{V,V}: F(V\otimes V) \cong F(V) \otimes F(V)$ the monoidal structure of $F$.  Using  $J$ repeatedly, we have
$$F(V^{\otimes n}) \cong F(V)^{\otimes n}$$ for each $n\in \mathbb{N}$.
Therefore we have $F(\bigoplus\limits_{n \in \mathbb{N}}V^{\otimes n})\cong \bigoplus\limits_{n \in \mathbb{N}}F(V^{\otimes n})\cong \bigoplus\limits_{n \in \mathbb{N}}F(V)^{\otimes n}$. That is $F(T(V)) \cong T(F(V))$.\par
    Recall that the  Nichols algebras in $\mathcal{C}$ are of the form $T(X)/I$ where $X \in \mathcal{C}$ and $I$ is the unique maximal graded Hopf ideal in $T(V)$ generated by homogeneous elements of degree greater than or equal to $2$.  According to $F(T(V)) \cong T(F(V))$, we have
$$F(T(V)/I) \cong T(F(V))/F(I),$$ which is finite-dimensional as well.
Here $F(I)$ is a homogeneous Hopf ideal of degree greater than or equal to $2$ of $T(F(V))$. Note that the Nichols algebra generated by $F(V)$ must be of the form $T(F(V))/J$, where $J$ is the unique maximal homogeneous graded Hopf ideal of $T(F(V)) \in \mathcal{D}$ with degree greater than or equal to $2$.  Hence   $ F(I) \subset J$ and $T(F(V))/J \subset T(F(V))/F(I)$.
On the other hand, $F^{-1}: \mathcal{D} \longrightarrow \mathcal{C}$ is the  inverse of $F$, which is an exact monoidal functor. Hence
$$F^{-1}(T(F(V))/J) \cong T(F^{-1}(F(V))/F^{-1}(J)\cong T(V)/F^{-1}(J) \supseteq T(V)/I.$$
So $F^{-1}(J)\subseteq I$, combining $F(I) \subseteq J$ implies $F(I)=J$, which leads to $F(T(V)/I) \cong T(F(V))/J$ and the proof is done. 
\end{proof}

\section{Monoidal equivalence between Yetter-Drinfeld module categories}

We build upon the theory of Hopf pairings and partial dualization in braided monoidal categories \cite{Simon}. We apply this theory to the setting of coquasi-Hopf algebras. Our main result, Theorem \ref{C-cor3.8}, establishes a key braided monoidal equivalence $\Omega$ between categories of Yetter-Drinfeld modules, which is important to our reflection theory.
\subsection{Hopf pairings  in braided monoidal categories}
In this subsection, $\mathcal{C}$ always represents a braided monoidal category with associative isomorphism $a$ and braiding isomorphism $c$ unless stated otherwise.
We begin by introducing the notion of Hopf pairings.
\begin{definition}
     Let \( A \) and \( B \) be Hopf algebras in \( \mathcal{C} \). A morphism \( \omega : A \otimes B \to 1 \) in \( \mathcal{C} \) is called a Hopf pairing, if the following identities hold:
\begin{align*}
\omega \circ (\mu_A \otimes \operatorname{id}_B)& = \omega \circ (\operatorname{id}_A \otimes \omega \otimes \operatorname{id}_B)\circ  (\operatorname{id}_A\otimes a^{-1}_{A,B,B}) \circ a_{A,A,B\otimes B}\circ (\operatorname{id}_{A \otimes A} \otimes \Delta_B), \\
\omega \circ (\eta_A \otimes \operatorname{id}_B) &= \varepsilon_B,\\
\omega \circ (\operatorname{id}_A \otimes \mu_B)& = \omega \circ (\operatorname{id}_A \otimes \omega \otimes \operatorname{id}_B)\circ  (\operatorname{id}_A\otimes a^{-1}_{A,B,B}) \circ a_{A,A,B\otimes B} \circ (\Delta_A \otimes \operatorname{id}_{B \otimes B}),\\
\omega \circ (\operatorname{id}_A \otimes \eta_B)& = \varepsilon_A.
\end{align*}
    A Hopf pairing $\omega : A \otimes B \to 1$ is called non-degenerate
if there is a morphism $\omega' : 1\rightarrow B \otimes A$, such that
$$ (\omega \otimes \operatorname{id}_A) \circ a^{-1}_{A,B,A}\circ (\operatorname{id}_A \otimes \omega')= \operatorname{id}_A, \ \ \text{and} \ \ (  \operatorname{id}_B\otimes \omega) \circ a^{-1}_{B,A,B}\circ ( \omega'\otimes \operatorname{id}_B)= \operatorname{id}_B.$$
\end{definition}
There are some useful properties of Hopf pairings, we list here for further use.
\begin{rmk}\textup{[\citealp{Simon}, Remark 2.13, 2.16]}\rm\label{C-rmk2.3}\begin{itemize}
    \item [(1)] Suppose  $\omega: A \otimes B \rightarrow \bfo$ is a Hopf pairing in $\mathcal{C}$, then \begin{equation}\label{C-3.0}
        \omega\circ (\operatorname{id} \otimes \mathcal{S}_B)=\omega\circ (\mathcal{S}_A\otimes \operatorname{id}).
    \end{equation} 
    \item [(2)] If $\omega: A\otimes B \rightarrow \textbf{1}$  is a Hopf pairing, then also \begin{equation}\label{C-3.15}
        \omega^{-}:=\omega \circ c^{-1}_{A,B}\circ (\mathcal{S}_B^{-1}\otimes \mathcal{S}_A^{-1}): B \otimes A \rightarrow \bfo.
    \end{equation} Furthermore, if $\omega$ is non-degenerate, then $\omega^{-}$ is non-degenerate as well with inverse copairing $(\mathcal{S}_A\otimes \mathcal{S}_B)\circ c_{B,A}\circ \omega'$.
    \item [(3)] If $A$ has a left dual object $A^*$,  then $A^*$ naturally carries the structure of a Hopf algebra in $\mathcal{C}$ such that the evaluation morphism $\operatorname{ev}_A:A^* \otimes A \rightarrow \textbf{1}$ is a Hopf pairing. Moreover, this Hopf pairing is non-degenerate with inverse copairing $\operatorname{coev}_A$.
\end{itemize}
\end{rmk}
Having established the properties of Hopf pairings, we now recall the  important  braided monoidal equivalence $\Omega^{\omega}$ induced by a certain Hopf pairing $\omega$.

\begin{thm}\textup{[\citealp{Simon}, Theorem 3.20]}\label{C-thm2.6}
Let $\omega: A \otimes B \to \mathbf{1}$ be a non-degenerate Hopf pairing. 
The categories $\AAC$ and $\BBC$ are braided monoidal equivalent via the monoidal 
functors $\Omega^{\omega}$ as follows.   
\begin{align*}
\Omega^{\omega}: \AAC &\longrightarrow \BBC,\\
(X,\rho_A,\delta_A)&\mapsto (X,\rho_B,\delta_B).
\end{align*}
Where $\rho_B, \delta_B$ are given by 
\begin{align*}
    \rho_B&=(B\otimes X \xrightarrow{c_{X,B}^{-1}} X\otimes B \xrightarrow{\delta_A \otimes \operatorname{id}_B}(A\otimes X)\otimes B \xrightarrow{a_{X,A,B}\circ (c^{-1}_{A,X}\otimes \operatorname{id}_B)} X\otimes (A\otimes B) 
    \xrightarrow{ \operatorname{id}_X \otimes (\omega \circ (\operatorname{id}_A \otimes (\mathcal{S}_B^{-1})^2)) }X)\\
    \delta_B&=(X \xrightarrow{\omega'\otimes \operatorname{id}_X}(B\otimes A)\otimes X \xrightarrow{a_{B,A,X}} B \otimes (A\otimes X)\xrightarrow{\operatorname{id}_B \otimes \rho_A} B\otimes X \xrightarrow{c_{X,B}\circ c_{B,X}\circ (\mathcal{S}_B^2\otimes \operatorname{id}_X)}B\otimes X).
\end{align*}
\end{thm}
\begin{rmk}\label{C-rmk2.7}\rm
\begin{itemize}
    \item [(1)] Starting from a non-degenerate Hopf pairing $\omega: A \otimes B \to \mathbf{1}$. There is  a non-degenerate Hopf pairing  $\omega^{-}:   B \otimes A \rightarrow \bfo $ by Remark \ref{C-rmk2.3}. Thus we obtain an equivalence:
\begin{equation}\label{C-2.1}
\Omega^{\omega^{-}}:\BBC \rightarrow \AAC.    
\end{equation}
\item [(2)] 
The braided monoidal functor 
$$\Omega^{\omega^{-}} \circ \Omega^{\omega}:\AAC \rightarrow \AAC$$
is isomorphic to the identity functor by \textup{[\citealp{Simon}, Proposition 3.22]}\label{C-prop2.8}.
\item [(3)] For completeness, we recall the construction of the functor $\Omega^{\omega}$, which is investigated in \textup{[\citealp{Simon}, Theorem 3.16, Corollary 3.19]}. 
Let $\omega: A \otimes B \to \mathbf{1}$ be a non-degenerate Hopf pairing. 
The functor $\Omega^{\omega}$ is given by
$$ \Omega^{\omega}:= T \circ D.$$
The braided monoidal equivalence $T$ is given by \begin{align*}
    T&: \mathcal{YD}_A^A(\mathcal{C}) \to \AAC \\
    T&(X, \rho^r, \delta^r) = (X, \rho^- \circ (\mathcal{S}_A^{-1} \otimes \mathrm{id}_X), (\mathcal{S}_A \otimes \mathrm{id}_X) \circ \delta^+),
\end{align*}
with \(\rho^- := \rho^r \circ c_{X,A}^{-1}\) and \(\delta^+ := c_{X,A} \circ \delta^r\). \\
Another  braided monoidal equivalence $D$ is given by $$
D: \AAC \to \BBRC, \ \ 
(X, \rho_X, \delta_X) \mapsto (X, \rho_{D(X)}, \delta_{D(X)})$$

with  
\begin{align*}
    \rho_{D(X)} &= (\mathrm{id} \otimes \omega)\circ a_{X,A,B} \circ (c_{A,X}^{-1} \otimes S^{-1})\circ a^{-1}_{A,X,B}  \circ ( \delta_X\otimes \operatorname{id}), \\
    \delta_{D(X)} &= c_{B,X} \circ (\mathcal{S}_B \otimes \rho_X) \circ a_{B,A,X}\circ (\omega' \otimes \mathrm{id}).
\end{align*}

\end{itemize}

\end{rmk}
\subsection{Partial dualization of a Hopf algebra in braided monoidal categories}
We now turn our attention to the construction of partial dualization. This construction, originally developed in [\citealp{Simon}, Theorem 4.4], provides a powerful tool for studying Hopf algebras in braided monoidal categories.\par Let us begin by recalling the setup. Consider a braided monoidal category $\mathcal{C}$ and a Hopf algebra $R$ in this category. A partial dualization datum is defined as follows.
\begin{definition}
    A partial dualization datum is denoted by $\mathcal{A}=(R\overset{\pi}{\rightarrow} A,B, \omega)$ in $\mathcal{C}$ if \par
$\cdot$ There is a Hopf algebra projection $\pi:R \rightarrow A$ to a Hopf subalgebra $A \subset R$ in $\mathcal{C}$. \par 
$\cdot$ There is a Hopf algebra $B$ with a non-degenerate Hopf pairing $\omega: A\otimes B \rightarrow \bfo$.
\end{definition} 
Now let $H$ be a coquasi-Hopf algebra with bijective antipode, and $\mathcal{C}=\HH$. Suppose $R$ is a Hopf algebra in $\mathcal{C}$ and $\mathcal{A}=(R \overset{\pi}{\rightarrow} A,B, \omega)$ is a partial dualization datum. By Lemma \ref{C-lem2.12}, there is a well-defined $K \in \AAC$ such that $R\cong K\#A$ in $\mathcal{C}$.
Combined with Theorem \ref{C-thm2.6} discussed in the previous subsection, the following lemma is immediate.
\begin{lemma}
    \label{C-thm2.8}
    With above notations, the equivalence of braided categories
    $$ \Omega^{\omega}: \AAC \rightarrow \BBC$$ 
    induces an equivalence of braided monoidal categories:
    \begin{equation}
        ^{R}_{R}\mathcal{YD}(\mathcal{C})\cong {^{K}_{K}\mathcal{YD}}(\AAC) \overset{\overline{\Omega}}{\longrightarrow}  {^{L}_{L}\mathcal{YD}}(\BBC)\cong {^{L\#B}_{L\#B}\mathcal{YD}}(\mathcal{C}).
    \end{equation}
    Here $L:= \Omega(K)$ is
  a Hopf algebra in $\BBC$.
\end{lemma}
\begin{proof}
    Since $\Omega^{\omega}$ is a braided monoidal equivalence, then $L:= \Omega(K)$ is a Hopf algebra in $\BBC$. Furthermore, $L \# B$ is a Hopf algebra in $\mathcal{C}$. The braided monoidal equivalences 
    $$  ^{R}_{R}\mathcal{YD}(\mathcal{C})\cong {^{K}_{K}\mathcal{YD}}(\AAC),$$
    $$  {^{L}_{L}\mathcal{YD}}(\BBC)\cong {^{L\#B}_{L\#B}\mathcal{YD}}(\mathcal{C})$$
    follow by Lemma \ref{C-thm1.12}. Each object in ${^{K}_{K}\mathcal{YD}}(\AAC)$ naturally belongs to $\AAC$, therefore the equivalence $\Omega^{\omega}$ naturally extends to an equivalence of braided monoidal categories.
    $$ \overline{\Omega}: {^{K}_{K}\mathcal{YD}}(\AAC) \longrightarrow {^{L}_{L}\mathcal{YD}}(\BBC).$$
\end{proof}

We now specialize to the  case which we concern.
\begin{lemma}\label{C-lem3.7}
    Let $H$ be a coquasi-Hopf algebra with bijective antipode.  Assume that $N \in \HH$ and  $\mathcal{B}(N)$ is finite-dimensional, then we have an isomorphism of Hopf algebras in $\HH$
    $$ \BN^{*\operatorname{op,cop}}\cong \bB(N^*).$$
   Consequently, the evaluation map $\operatorname{ev}: \bB(N^*) \otimes \BN \rightarrow \textbf{1}$ is a non-degenerate Hopf pairing.
\end{lemma}
\begin{proof}
We establish this result in several steps. First, the object $\BN^{*\operatorname{op,cop}}$ with the following operations is a Hopf algebra in $\HH$ by [\citealp{AGn99}, Lemma 2.2.3]
$$ (\BN^*, \mu_{{\BN^*}}, c^{-1}_{{\BN^*},{\BN^*}}, \eta_{{\BN^*}}, c_{{\BN^*},{\BN^*}}\Delta_{{\BN^*}}, \varepsilon_{{\BN^*}}, \mathcal{S}_{{\BN^*}}).$$ 
 It is immediate that $\BN^*(1)=N^*$, thus $\BN^{*\operatorname{op,cop}}$ is generated by $N^*$. On the other hand, since $N^* \in \HH$, $\bB(N^*)$ is a Nichols algebra generated by $N^*$ in $\HH$. By universal property of Nichols algebra, there is a surjective Hopf algebra map $$\pi:\BN^{*\operatorname{op,cop}} \longrightarrow \bB(N^*)$$ in $\HH$. The object   $\bB(N^*)$ is the graded dual of $\bB(N)$ by [\citealp{defofnichols}, Lemma 2.6]. Therefore $\operatorname{dim}(\bB(N^*))=\operatorname{dim}(\bB(N)^*)$ since $\operatorname{dim}(\bB(N))$ is finite-dimensional. As a result 
  $$ \BN^{*\operatorname{op,cop}}\cong \bB(N^*).$$ as Hopf algebras in $\HH$.
  By Remark \ref{C-rmk2.3}(3),  the evaluation map $\operatorname{ev}: \bB(N)^{*\operatorname{op},\operatorname{cop}} \otimes \BN \rightarrow \textbf{1}$ is a non-degenerate Hopf pairing. Then by an appropriate abuse of notation, we can write $$\operatorname{ev}:\bB(N^*) \otimes \BN \rightarrow \textbf{1},$$ which confirms that this is indeed a non-degenerate Hopf pairing. 
\end{proof}
This lemma will play an important role in what follows, as noted in the subsequent application of Lemma \ref{C-thm2.8}.
\begin{thm}\label{C-cor3.8}
 Let $H$ be a coquasi-Hopf algebra with bijective antipode, and $\mathcal{C}=\HH$. Suppose $N=\bigoplus_{i \in I}N_i \in \HH$, and the Nichols algebras $\bB(N_i)$ are all finite-dimensional. For each $i\in I$, we have a partial dualization datum $\mathcal{A}_i=(\bB
 (N)\overset{\pi_i}{\rightarrow}\bB(N_i),\bB(N_i^*), \operatorname{ev}_{\bB(N_i)}^-)$.  We set $K_i:= \BN^{\operatorname{co}\bB(N_i)} \in {^{\bB(N_i)}_{\bB(N_i)}\mathcal{YD}(\mathcal{C})}$.  Then we have the following equivalences of braided monoidal categories 
 \begin{align*}
  \Omega_i:{^{\bB(N_i)}_{\bB(N_i)}\mathcal{YD}(\mathcal{C})}& \cong {^{\bB(N_i^*)}_{\bB(N_i^*)}\mathcal{YD}(\mathcal{C})},\\
     \widetilde{\Omega_i}:{^{K_i}_{K_i}\mathcal{YD}}({^{\bB(N_i)}_{\bB(N_i)}\mathcal{YD}(\mathcal{C})})&\cong{^{\Omega_i^{}(K_i)}_{\Omega_i^{}(K_i)}\mathcal{YD}}( {^{\bB(N_i^*)}_{\bB(N_i^*)}\mathcal{YD}(\mathcal{C})}).
 \end{align*}
\end{thm}
\begin{proof}
    It follows from Remark \ref{C-rmk2.7}(1), Lemma \ref{C-thm2.8} and \ref{C-lem3.7} immediately.
\end{proof}
\subsection{The functor $\Omega^{\omega}$  under $\mathbb{Z}$-gradings}
We now investigate the properties of the functor $\Omega^{\omega}$.\par 
Let $\NN\Gr \Mod_{\mathbbm{k}}$(resp.$\ZZ\Gr\Mod_{\mathbbm{k}}$) be the category of $\mathbb{N}_0$-graded vector spaces(resp.$\mathbb{Z}$-graded vector spaces).
The functor $\NN\Gr \Mod_{k} \rightarrow \ZZ\Gr \Mod_{k}$, which extends the $\NN$-grading of an object $V$ in $\NN\Gr \Mod_{k}$ to a $\ZZ$-grading by setting $V(n) = 0$ for all $n < 0$, is monoidal.  Consequently, we can view $\NN$-graded coalgebras and coquasi-Hopf algebras as $\ZZ$-graded coalgebras and $\ZZ$-graded coquasi-Hopf algebras, respectively.\par 
Furthermore, suppose $H$ is a $\mathbb{N}_0$-graded coquasi-Hopf algebra. A $\ZZ$-graded Yetter-Drinfeld module $V$ over $H$ is by definition an object $V$ in $\HH(\ZZ\Gr \Mod_{k})$. In other words, $V$ is a $\ZZ$-graded vector space such that the operation $H\otimes V \rightarrow V$ and comodule structure maps  $V \rightarrow H \otimes V$ are $\mathbb{Z}$-graded.
\begin{lemma}\label{C-lem 1.24}Suppose $H$ is a $\mathbb{Z}$-graded coquasi-Hopf algebra.\par 
\text{$(1)$} Let $K$ be a Nichols algebra in $\HH(\Zgr)$, and $K(1)=\oplus_{\gamma \in \mathbb{Z}} K(1)_\gamma$ a $\ZZ$-graded object in $\HH(\Zgr)$. Then there is a unique $\ZZ$-grading on $K$ extending the grading on $K(1)$. Moreover, $K(n)$ is $\ZZ$-graded in $\HH(\Zgr)$ for all $n \geq 0$.\par
\text{$(2)$}  Let $K$ be a $\ZZ$-graded braided Hopf algebra in $\HH(\ZZ\Gr \Mod_{k})$. Then the bosonization $K\#H$ is a $\ZZ$-graded Hopf algebra with $\deg K(\gamma)\#H(\lambda) = \gamma + \lambda$ for all $\gamma, \lambda \in \ZZ$.\par
\text{$(3)$}  Let $H_0 \subseteq H$ be a coquasi-Hopf subalgebra of degree $0$, and $\pi : H \rightarrow H_0$ a coquasi-Hopf algebra map with $\pi|_{H_0} = \mathrm{id}$. Define $R = H^{\operatorname{co}H_0}$. Then $R$ is a $\ZZ$-graded braided Hopf algebra in ${^{ H_0}_{H_0}\mathcal{YD}}(\ZZ\Gr \Mod_{k})$ with $R(\gamma) = R \cap H(\gamma)$ for all $\gamma \in \ZZ$.
\end{lemma}
\begin{proof}
(1) 
The module and comodule maps of $K(1)$ are $\ZZ$-graded and hence the infinitesimal braiding $c \in \mathrm{Aut}(K(1) \otimes K(1))$, being determined by the module and comodule maps, is $\ZZ$-graded. Moreover, the structure maps of $K(n)$ for $n \in \mathbb{N}$  are determined by $c$ and $K(1)$, since $K(1)$ generates $K$.  Therefore $K$ is a $\mathbb{Z}$-graded objects in $\HH(\Zgr)$.\par 
(2) and (3) are the same as the $\mathbb{N}_0$-graded case.
\end{proof}

Let $A,B \in \HH$ be $\mathbb{N}_0$-graded Hopf algebras with a non-degenerate Hopf pairing $\omega$. Moreover, 
$$ \omega(A(m),B(n))=0, \ \ \text{for all} \ m\neq n.$$
Obviously, if $\BN$ is finite-dimensional and let $A=\bB(N^*)$, $B=\BN$, $\omega=\operatorname{ev}$. This setting satisfies the above requirement. To extend this construction to the  $\mathbb{Z}$-graded setting, we extend the $\mathbb{N}_0$-grading to $\mathbb{Z}$-grading  by
$$ A(n)=0, \ B(n)=0 \ \  \text{for all}  \ n<0.$$
We then define a new $\mathbb{Z}$-grading on $A$ by $\operatorname{deg}(A(n))=-n$ for all $n \in \mathbb{Z}$, which ensures that  $\omega: A \otimes B \rightarrow \bfo$ is $\mathbb{Z}$-graded. Here, the grading of $\mathbbm{k}$ is given by $\mathbbm{k}(n) = 0$, if $n \neq  0$, and $\mathbbm{k}(0) = \mathbbm{k}$.\par 
Now we return to the purpose of this subsection.
We first introduce the following notations for further use.
\begin{definition}
Let $H$ be a coquasi-Hopf algebra with bijective antipode, and $\mathcal{C}=\HH$.
Suppose $R$ is an $\mathbb{N}_0$-graded Hopf algebra in  $\HH$. Let $X$ be an  $R$-module in $\mathcal{C}$ and $Y$ be an $R$-comodule in $\mathcal{C}$. For $n \geq 0$, we define:
\begin{align*}
\mathcal{F}_nX &= \{x \in X \mid R(i) \rhd x = 0 \text{ for all } i > n\}, \ \\
\mathcal{F}^nY &= \{y \in Y \mid \delta_Y(y) \in \bigoplus_{i=0}^n R(i) \otimes Y\}. 
\end{align*}
\end{definition}
 With the above notations, the functor $\Omega$ relates the operator $\mathcal{F}_n$ to $\mathcal{F}^n$.
\begin{lemma}\label{C-lemOmega}
    Let $(\bB(N^*),\BN,\operatorname{ev})$ be a Hopf pairing in $\HH$ with the equivalence of tensor category:
    $$ \Omega:{^{\bB(N)}_{\bB(N)}\mathcal{YD}(\mathcal{C})} \cong {^{\bB(N^*)}_{\bB(N^*)}\mathcal{YD}(\mathcal{C})} .$$ Let $V \in {^{\bB(N)}_{\bB(N)}\mathcal{YD}(\mathcal{C})}$,\par 
    (1) $\mathcal{F}_n \Omega(V)=\mathcal{F}^nV$ for all $n\geq 0$.\par (2) $\mathcal{F}^n\Omega(V)=\mathcal{F}_n(V)$ for all $n\geq 0$.\par
    (3) Suppose  $V$ is  a $\mathbb{Z}$-graded object in ${^{\bB(N)}_{\bB(N)}\mathcal{YD}(\mathcal{C})}$ , then $\Omega(V)$ is a $\mathbb{Z}$-graded object in ${^{\bB(N^*)}_{\bB(N^*)}\mathcal{YD}(\mathcal{C})} $, where the grading of  $\Omega(V)$ is given by $\Omega(V)(n)=V(-n)$ for all $n\in \mathbb{Z}$.
\end{lemma}
\begin{proof}
Recall that the braided monoidal equivalence of abelian category 
$$\Omega: {^{\bB(N)}_{\bB(N)}\mathcal{YD}(\mathcal{C})} \cong {^{\bB(N^*)}_{\bB(N^*)}\mathcal{YD}(\mathcal{C})}, $$ sends 
$$ (V,\rho_{\bB(N)},\delta_{\bB(N)}) \mapsto (V,\rho_{\bB(N^*)},\delta_{\bB(N^*)}).$$
It is induced by the non-degenerate Hopf pairing 
$$ \operatorname{ev}^{-}:=\operatorname{ev}\circ c^{-1}_{\bB(N^*),\bB(N)}\circ (\mathcal{S}^{-1}_{\bB(N)}\otimes 
\mathcal{S}^{-1}_{\bB(N^*)}):\BN \otimes \bB(N^*)\rightarrow \mathbbm{k}.$$\par 
(1) By naturality of braiding and associative isomorphism,
\begin{align*}
    &\rho_{\bB(N^*)}\\&=(\operatorname{id}_V\otimes (\operatorname{ev}^{-}\circ (\operatorname{id}_{\bB(N)}\otimes (\mathcal{S}_{\bB(N^*)}^{-1})^2)))\circ a_{V,\bB(N),\bB(N^*)}\circ (c^{-1}_{\mathcal{B}(N),V}\otimes \operatorname{id}_{\mathcal{B}(N^*)}) \circ (\delta_{\bB(N)}\otimes \operatorname{id}_{\mathcal{B}(N)}) \circ c^{-1}_{V,\mathcal{B}(N^*)}\\
    &\overset{(\ref{C-3.0})}{=}(\operatorname{id}_V\otimes \operatorname{ev}^{-}\circ c^{-1}_{\bB(N),\bB(N^*)} \circ (\mathcal{S}_{\bB(N^*)}^{-1}\otimes \mathcal{S}_{\bB(N)}^{-1}))\circ c^{-1}_{V,\bB(N^*)\otimes \bB(N)}\circ a^{-1}_{\bB(N^*),\bB(N),V} \circ (\operatorname{id}_{\bB(N^*)}\otimes \delta_{\bB(N)})\\
    &\overset{(\ref{C-3.15})}{=}(\operatorname{id}_V\otimes (\operatorname{ev}^{-})^{-})\circ c^{-1}_{V,\bB(N^*)\otimes \bB(N)}\circ a^{-1}_{\bB(N^*),\bB(N),V} \circ (\operatorname{id}_{\bB(N^*)}\otimes \delta_{\bB(N)})\\
    &=((\operatorname{ev}^{-})^{-}\otimes \operatorname{id}_V)\circ a^{-1}_{\bB(N^*),\bB(N),V} \circ (\operatorname{id}_{\bB(N^*)}\otimes \delta_{\bB(N)}).
\end{align*}
Here $(\operatorname{ev}^{-})^{-}=\operatorname{ev}\circ c^{-1}_{\bB(N^*),\bB(N)}\circ c^{-1}_{\bB(N),\bB(N^*)}\circ ((\mathcal{S}^{-1}_{\bB(N^*)})^2\otimes 
(\mathcal{S}^{-1}_{\bB(N)})^2)$.\par 
     Define an auxiliary map
    $\rho':=(\operatorname{ev}\otimes \operatorname{id}_V) \circ a^{-1}_{\bB(N^*),\bB(N),V} \circ (\operatorname{id}_{\bB(N^*)}\otimes \delta_{\bB(N)})$, which is derived from $\rho_{\bB(N^*)}$.  By [\citealp{rootsys}, Corollary 3.3.4, Lemma 12.2.11], 
    $(V, \rho')$ is still a $\bB(N^*)$-module in $\mathcal{C}$ and
    $$\mathcal{F}_n\Omega(V)=\mathcal{F}_n(V,\rho').$$
    Now for fixed $n \geq 0$, the  kernel of the induced map:
    \begin{equation}\label{C-2.3}
        \bB(N^*)\otimes V \longrightarrow \operatorname{Hom}(\bB(N)(n),V), \ \ a\otimes v\mapsto(b \mapsto \Phi(a_{-1},b_{-1},v_{-1})^{-1}\langle a_0,b_0\rangle v_0)
    \end{equation} 
    is $\bigoplus_{m\neq n}\bB(N^*)(m) \otimes V$. This means,  if $v \in \mathcal{F}^n V$, then for each $i >n$, $\rho'(\bB(N^*)(i)\otimes v)=0$ by (\ref{C-2.3}). On the other hand, if $v \in \mathcal{F}_n(V,\rho')$, this will imply $\delta(v)\in \bigoplus_{i=0}^n\BN(i)\otimes V$,  which completes the proof of (1): $\mathcal{F}_n \Omega(V)=\mathcal{F}_n(V,\rho')=\mathcal{F}^nV$.\par
  (2)  For the second statement, recall $\delta_{\bB(N^*)}$ is given by 
    $$ \delta_{\bB(N^*)}=(\mathcal{S}_{\bB(N^*)}^2\otimes \operatorname{id}_V) \circ c_{{\bB(N^*)},V} \circ  c_{V,{\bB(N^*)}} \circ  (\operatorname{id}_{\bB(N^*)} \otimes \rho_{\bB(N)}) \circ a_{{\bB(N^*)},{\bB(N)},V} \circ ((\operatorname{ev}^-)'\otimes \operatorname{id}_V ). $$
    We denote $\delta'=(\operatorname{id}_{\bB(N^*)} \otimes \rho_{\bB(N)}) \circ a_{{\bB(N^*)},{\bB(N)},V} \circ ((\operatorname{ev}^-)'\otimes \operatorname{id}_V ) $, which origins from $\delta_{\bB(N)}$. Again by [\citealp{rootsys}, Corollary 3.3.6, Lemma 12.2.11]
$$\mathcal{F}^n(V,\delta')=\mathcal{F}^n\Omega(V).$$
 A similar argument to that in (1) shows that $\mathcal{F}_n(V)=\mathcal{F}^n(V,\delta')$.  This combines together shows  $\mathcal{F}_n(V)=\mathcal{F}^n(V,\delta')=\mathcal{F}^n\Omega(V)$.\par 
(3) We regard $\bB(N)$ and $\bB(N^*)$ as $\mathbb{Z}$-graded Hopf algebras in $\GG$. Let $V=\bigoplus_{n \in \mathbb{Z}}V(n)$ be its gradation, and set  $\Omega(V)(n)=V(-n)$ for all $n \in \mathbb{Z}$. Then  for $m\geq 0$, by definition of $\rho_{\bB(N^*)}$, we have
$$ \rho_{\bB(N^*)} ( \bB(N^*)(m)\otimes \Omega(V)(n) )=\rho_{\bB(N^*)} (\bB(N^*)(m)\otimes V(-n))\subseteq V(-m-n)=\Omega(V(m+n)).   $$
On the other hand, by definition of $\delta_{\bB(N^*)}$
$$ \delta_{\bB(N^*)}(\Omega(V)(n))=\delta_{\bB(N^*)}(V(-n))\subseteq \bigoplus_{i\in \mathbb{Z}}\bB(N^*)(i)\otimes V(i-n)=\bigoplus_{i \in \mathbb{Z}} \bB(N^*)(i)\otimes \Omega(V(n-i)).$$
The above two equations imply that
 $\Omega(V)$ is a well-defined  $\mathbb{Z}$-graded object in ${^{\bB(N^*)}_{\bB(N^*)}\mathcal{YD}(\mathcal{C})}$.
\end{proof}

\section{Projections of Nichols algebras }
Let $(\mathbbm{k}G,\Phi)$ be a pointed coquasi-Hopf algebra, where $G$ is a finite group, $\Phi$ is a $3$-cocycle on $G$. In this section, we study projection of Nichols algebras in $\GG$, which will play an important role when considering reflection of simple Yetter-Drinfeld modules. The proofs of  some lemmas and theorems in this section, while are conceptually parallel to those in the Hopf algebra setting \cite{HSdualpair}, take a different path due to the non-strict monoidal structure.
\subsection{The structure  of the space of coinvariant $K$}
 To study the projection properties of Nichols algebras, we must first examine their structure in the category $\GG$. \par For any $M,N \in \GG$, there is a canonical surjection of braided Hopf algebra in $\GG$
$$ \pi_{\bB(N)}:\bB(M\oplus N) \rightarrow \bB(N), \  \  \ \pi_{\bB(N)}\mid_N=\operatorname{id}, \ \ \ \pi_{\bB(N)}\mid_M=0. $$  This surjection naturally induces a canonical projection of coquasi-Hopf algebras:
$$ \pi=\pi_{\bB(N)}\#\operatorname{id}_{\kG}:\bB(M\oplus N)\#\kG\rightarrow \bB(N)\# \kG.$$
To simplify our notation, we introduce the following conventions:
$$ \mathcal{A}(M \oplus N):=\bB(M\oplus N)\#\kG,\ \mathcal{A}(N):=\bB(N)\# \kG.$$
Building upon this projection structure, we can now construct the space of coinvariants. There is a natural injection $\iota: \mathcal{A}(N)\rightarrow \mathcal{A}(M\oplus N)$ such that $\pi\circ \iota=\operatorname{id}_{\mathcal{A}(N)}$. \par 
Let 
$$K= \mathcal{A}(M \oplus N)^{\operatorname{co}\AN}$$
be the space of right $\AN$-coinvariant elements with respect to the projection $\pi$. By virtue of Lemma \ref{C-Lemma 1.6}, 
$K$ inherits the structure of a braided Hopf algebra in $\BNG$ 
with $\AN$-coaction:
$$ \delta: K \rightarrow \AN \otimes K , \ \ \  X_{-1} \otimes X_0=:\delta(X):=\pi(X_1)\otimes X_2,$$
and  $\AN$-action:
\begin{equation}\label{C-3.1}
    \operatorname{ad}: \AN \otimes K \rightarrow K, \ \ \ \  A\otimes X \mapsto \operatorname{ad}(A)(X)=\Phi(A_1X_{-1}, \mathbb{S}(A_3)_1,A_4)(A_2X_0)\mathbb{S}(A_3)_2.
\end{equation}
   Here, $\mathbb{S}$ denotes preantipode of $\AN$.\par 

Denote $\pi_{\kG}: \AMN\rightarrow \kG$, and $\pi'_{\kG}:\AN\rightarrow \kG$ be the natural projection. These maps satisfy a compatibility condition, namely, there exists a commutative diagram relating these projections:
\begin{center}
	\begin{tikzpicture}
		\node (A) at (-2,0) {$\AMN$};
		\node (B) at (3,0) {$\AN$};
		\node (C) at (3,-2) {$\kG$};
		\node (D) at (3,2) {$\AN$};
  \draw[->] (A) --node [above ] {$\pi$} (B);
		\draw[->] (B) --node [ right] {$\pi'_{\kG}$} (C);	
		\draw[->] (A) --node [below ] {$\pi_{\kG}$} (C);
  \draw[->](D)--node [above left]{$\iota$}(A);
  \draw[->](D)--node[right]{$=$}(B);
	\end{tikzpicture}
\end{center}

   \begin{lemma}
     The space $K$ coincides with $\mathcal{B}(M \oplus N)^{\text{co} \mathcal{B}(N)}$, which is precisely the space of right $\mathcal{B}(N)$-coinvariant elements with respect to $\pi_{\mathcal{B}(N)}$.
   \end{lemma}
   \begin{proof}
       Let $X \in K$, then $X_1 \otimes \pi_{\kG}(X_2)=X_1\otimes \pi'_{\kG}\circ \pi(X_2)=X\otimes 1$. Hence $X \in  \AMN^{\operatorname{co}\kG}=\bB(M\oplus N)$. Now, we assume $X$ is homogeneous. Under this assumption, we have
       $$ X\otimes 1=(\operatorname{id}\otimes \pi)\Delta_{\AMN}(X)=X^1\#x^2 \otimes \pi(X^2)=X^1\#x^2 \otimes \pi_{\bB(N)}(X^2).$$
       On the other hand, consider the projection $\gamma=\operatorname{id}\# \varepsilon:\AMN \rightarrow \bB(M \oplus N)$. Applying this projection yields
        $$ 
       X \otimes 1=(\gamma\otimes \operatorname{id} )(X \otimes 1)=X^1 \otimes \pi_{\bB(N)}(X^2)=(\operatorname{id}\otimes \pi_{\bB(N)})(\Delta_{\bB(M\oplus N)}(X)).$$
       This demonstrates that $X \in \bB(M\oplus N)^{\operatorname{co}\bB(N)}$, completing one direction of the proof.\par 
       Conversely, let $X \in \bB(M \oplus N)$ be a homogeneous element with $$ (\operatorname{id}\otimes \pi_{\bB(N)})\Delta_{\bB(M\oplus N)}(X)=X^1 \otimes \pi_{\bB(N)}(X^2)=X \otimes 1.$$
       Since $\pi_{\bB(N)}$ is $\kG$-colinear, we can deduce that
       $$  X\otimes 1 \otimes 1=X^1 \otimes x^2 \otimes \pi_{\bB(N)}(X^2).$$
       which implies $ X \otimes 1=X^1\#x^2 \otimes \pi_{\bB(N)}(X^2).$
       Furthermore, observe that $$(\operatorname{id}\otimes \pi)\Delta_{\AMN}(X\#1)=X^1\#x^2\otimes \pi(X^2)=X^1\#x^2\otimes \pi_{\bB(N)}(X^2)=X\otimes 1.$$
       Therefore, we conclude that $X \in K$, establishing the reverse inclusion.
       
   \end{proof}
Having characterized the coinvariant space $K$, we now turn our attention to its primitive elements. Indeed, the space of primitive elements of $K$ remains an object in $\BNG$.
\begin{lemma}\label{C-lem3.9}
    The space of primitive elements $P(K)=\{ x\in K \mid \Delta_K(x)=x\otimes 1+1\otimes x\}$ is a subobject of $K$ in $\BNG$.
\end{lemma}
\begin{proof}
Let $\Delta_K: K\rightarrow K\otimes K$ be its comultiplication map and $\widetilde{\Delta}:K \rightarrow K \otimes K, \ x \mapsto 1\otimes x+x \otimes 1$.
    We consider the following map 
    $$ \widehat{\Delta}:K\rightarrow K\otimes K,\ \ \  x \mapsto \Delta_K(x)-\widetilde{\Delta}(x).$$
    $\widehat{\Delta}$ is a comodule map since  both $\Delta_K$    and 
    $\widetilde{\Delta}$ is colinear. Furthermore, for all $A \in \BNG$ and $X \in K$, \begin{align*}
        \widehat{\Delta}(\operatorname{ad}(A)(X))&=\Delta_K(\operatorname{ad}(A)(X))-\operatorname{ad}(A)(X)\otimes 1-1\otimes\operatorname{ad}(A)(X)\\&=\operatorname{ad}(A)(\Delta_K(X))-\operatorname{ad}(A)(1\otimes X+X\otimes 1)
    \\&=\operatorname{ad}(A)(\Delta_K(X)-\widetilde{\Delta}(X)).
    \end{align*}
    Therefore $\widehat{\Delta}$ is a morphism in $\BNG$.
    We conclude that $P(K)$ is a subobject of $K$ in $\BNG$, because $P(K)$ is precisely the kernel of the map $\widehat{\Delta}$.
\end{proof}

With the structural properties of $P(K)$ established, we now investigate the explicit form of $K$ through its relation to the adjoint map. Since $K = \mathcal{B}(M \oplus N)^{\text{co}\mathcal{B}(N)}$, it is natural to consider $\text{ad}(\mathcal{B}(N))(M)$.

\begin{lemma}\label{C-lem3.2}
    Let $M,N \in \GG$, $K=(\AMN)^{\operatorname{co}\AN}$.
   The standard $\mathbb{N}_0$-grading of $\bB(M\oplus N)$ induces an $\mathbb{N}_0$-grading on 
    $$L:=\operatorname{ad}(\bB(N))(M)=\bigoplus_{n \in \mathbb{N}}\operatorname{ad}(N)^{n}(M)$$
    with degree $\operatorname{deg}(\operatorname{ad}(N)^n(M))=n+1$. Then $L$ is a $\mathbb{N}_0$-graded object in $\BNG$. Moreover, $L \subseteq K$ is a subobject of $K$ in $\BNG$.\par 
\end{lemma}
\begin{proof}
     Let $A \in N$ and $X \in \bB(M \oplus N)$ be homogeneous elements.
By equation (\ref{C-1.27}), we have
    $$ \operatorname{ad}(A)(X)=AX-\operatorname{ad}(a)(X)A.$$
The element $\operatorname{ad}(A)(X)$ is of degree $\operatorname{deg}(X)+1$ in the Nichols algebra $\mathcal{B}(M\oplus N)$. Moreover, we have $ \operatorname{ad}(N)^{n}(M) \cap \operatorname{ad}(N)^m(M)=0$ for $ n\neq m$. This is implies the direct sum decomposition of $L$. \par 
Since $M \subseteq K$ and $K \in \BNG$, we conclude that $L=\operatorname{ad}(\BN)(M)\subseteq K$.
Note that\\ $\operatorname{ad}(\BN)(M)\subseteq \operatorname{ad}(\AN)(M).$ Conversely, since $M \in \GG$, we have 
    $$ \operatorname{ad}(\AN)(M) \subseteq \operatorname{ad}(\BN)(\kG \rhd M) \subseteq \operatorname{ad}(\BN)(M).$$
  Hence $L=\operatorname{ad}(\AN)(M)$. \par 
 It is clear that the operation 
    $$ \operatorname{ad}:\AN \otimes L \rightarrow L$$ is well-defined since  $\operatorname{ad}$ satisfies the first  axiom of Yetter-Drinfeld modules.
Furthermore, the operation
$ \operatorname{ad}:\AN \otimes L\rightarrow L$ is $\mathbb{N}_0$-graded. 
Next, we want to show $L$ is a $\AN$-comodule. The comodule structure is given by 
$ \delta_K$. 
Now we assume $A \in \BN$ and $X \in M$ are homogeneous elements. A direct computation shows:
\begin{align*}
\delta_K(\operatorname{ad}(A)(X))&=p((\operatorname{ad}(A_3)(X))_{-1},A_4)q(A_1x,\mathcal{S}(A_6))(A_2x)\mathcal{S}(A_5)\otimes (\operatorname{ad}(A_3)(X))_0.
\end{align*}
Note that the condition 
$$  \operatorname{ad}(A_3)(X)_{-1}\notin G \Longrightarrow p((\operatorname{ad}(A_3)(X))_{-1},A_3) = 0.$$
Therefore, since $\operatorname{ad}(A_3)(X)$ is homogeneous,  we have
$$\delta_K(\operatorname{ad}(A)(X))=p(\operatorname{deg}(\operatorname{ad}(A_3)(X)),A_4)q(A_1x,\mathcal{S}(A_6))(A_2x)\mathcal{S}(A_5)\otimes (\operatorname{ad}(A_3)(X)),   $$
where $\operatorname{deg}(\operatorname{ad}(A_2)(X)) \in G$ represents its associated degree.
This shows $\delta_K:L  \rightarrow \AN \otimes L$ is well-defined. The third axiom holds automatically since $L \subseteq K$. Thus $L$ is a $\mathbb{N}_0$-graded object in $\BNG$ and  is a subobject of $K$.\par
\end{proof}

\begin{lemma}\label{C-lem3.3}
    For all $X \in L$, we have  $\Delta_{\bB(M\oplus N)}(X) - X \otimes 1 \in \bB( N)\otimes L$. \par 
\end{lemma}
\begin{proof} Recall $L=\bigoplus_{n \in \mathbb{N}}\operatorname{ad}(N)^{n}(M)$.
 We proceed by induction on $n$. Now let $A \in N$ and $X \in M$ be homogeneous elements. A direct computation yields:
\begin{align*}
\Delta_{\BMN}(\operatorname{ad}(A)(X))&=\Delta_{\BMN}(AX-(\operatorname{ad}(a)(X)A))\\
&=AX\otimes 1+1\otimes AX+A\otimes X+\operatorname{ad}(a)(X)\otimes A\\&-\operatorname{ad}(a)(X)A\otimes 1-1\otimes \operatorname{ad}(a)(X)A-\operatorname{ad}(a)(X)\otimes A-\operatorname{ad}(axa^{-1})(A)\otimes \operatorname{ad}(a)(X)\\
&=\operatorname{ad}(A)(X)\otimes 1+1\otimes\operatorname{ad}(A)(X)+A\otimes X-\operatorname{ad}(axa^{-1})(A)\otimes \operatorname{ad}(a)(X).
\end{align*}
Note that $X=\operatorname{ad}(1)(X) \in L$,
therefore $\Delta_{\BMN}(\operatorname{ad}(A)(X))- \operatorname{ad}(A)(X)\otimes 1\in \BN \otimes L$.\par 
For a fixed $m \in \mathbb{N}_0$. We assume that for all $Y \in \operatorname{ad}(N)^{m'}(M)$, $1\leq m' \leq m$, we have 
 $$ \Delta_{\bB(M\oplus N)}(Y) - Y \otimes 1 \in \bB( N)\otimes L.$$
 Now let  $A \in N $ and  $X \in \operatorname{ad}(N)^{m}(M)$ be homogeneous, we have
\begin{align*}
    &\Delta_{\BMN}(\operatorname{ad}(A)(X))=\Delta_{\BMN}(AX-(\operatorname{ad}(a)(X))A)\\
    &=(A \otimes 1+1 \otimes A)(X^1 \otimes X^2)\\&-\frac{\Phi(a,x^1,x^2)\Phi(ax^1a^{-1},ax^2a^{-1},a)}{\Phi(ax^1a^{-1},a,x^2)}(\operatorname{ad}(a)(X^1)\otimes \operatorname{ad}(a)(X^2))(1\otimes A+A \otimes 1)\\
&=\operatorname{ad}(A)(X)\otimes 1+ \sum_{X^1 \otimes X^2 \neq X \otimes 1}\Phi(a,x^1,x^2)AX^1\otimes X^2+\frac{\Phi(a,x^1,x^2)}{\Phi(ax^1a^{-1},a,x^2)}\operatorname{ad}(a)(X^1)\otimes AX^2\\&-\frac{\Phi(a,x^1,x^2)}{\Phi(ax^1a^{-1},a,x^2)}\operatorname{ad}(a)(X^1)\otimes \operatorname{ad}(a)(X^2)A\\ &-\frac{\Phi(a,x^1,x^2)\Phi(ax^1a^{-1},ax^2a(ax^2)^{-1},a)}{\Phi(ax^1a^{-1},a,x^2)}\operatorname{ad}(a)(X^1)\operatorname{ad}(ax^2a^{-1}) (A)\otimes \operatorname{ad}(a)(X^2)
\\
&=\operatorname{ad}(A)(X)\otimes 1+\sum_{X^1 \otimes X^2 \neq X \otimes 1}\Phi(a,x^1,x^2)AX^1\otimes X^2+\frac{\Phi(a,x^1,x^2)}{\Phi(ax^1a^{-1},a,x^2)}\operatorname{ad}(a)(X^1)\otimes \operatorname{ad}(A)(X^2)\\
&-\frac{\Phi(a,x^1,x^2)\Phi(ax^1a^{-1},ax^2a(ax^2)^{-1},a)}{\Phi(ax^1a^{-1},a,x^2)}\operatorname{ad}(a)(X^1)\operatorname{ad}(ax^2a^{-1}) (A)\otimes \operatorname{ad}(a)(X^2).
\end{align*}
By assumptions, $$ \Delta_{\bB(M\oplus N)}(X) - X \otimes 1 \in \bB( N)\otimes L.$$ Therefore the second term lies in $\BN \otimes{ L}$.
Since $A \in N$, $\operatorname{ad}(a)(X^1)\in \BN$ and 
$$\operatorname{ad}(A)(X^2)\in \operatorname{ad}(N)(L)=\operatorname{ad}(N)(\bigoplus_{n\in \mathbb{N}_0}\operatorname{ad}(N)^n(M)) \in L,$$
the third term lies in $\BN \otimes{ L}$.
For the last term,  by Lemma \ref{C-lem3.2}, we have  $L \in \BNG$. Hence it belongs to $\BN \otimes{ L}$ as well, which completes the induction and proves the lemma.

    \end{proof}
\subsection{The semisimplicity of  the object $L$}
We now turn to the structural analysis of the object $L$, whose properties are useful to our subsequent arguments regarding Nichols algebras.\par 
Recall that a graded coalgebra is a coalgebra \( C \) provided with a grading \( C = \bigoplus_{m \in \mathbb{N}_0} C^m \) such that \( \Delta(C^m) \subseteq \bigoplus_{i+j=m} C^i \otimes C^j \). 
Let \( \Delta_{i,j}: C^m \rightarrow C^i \otimes C^j \) denote the map \( \text{pr}_{i,j} \Delta \), where \( m = i + j \) and $\operatorname{pr}_{i,j}:C \otimes C \rightarrow C^i\otimes C^j$. More generally, if \( i_1, \ldots, i_n \in \mathbb{N}_0 \) and \( i_1 + \cdots + i_n = m \), then \( \Delta_{i_1, \ldots, i_n} \) is defined by the commutative diagram
  \begin{center}
	\begin{tikzpicture}
		\node (A) at (0,0) {$C^m$};
		\node (B) at (5,0) {$\bigoplus\limits_{j_1+j_2+\cdots+j_n=m}C^{j_1}\otimes C^{j_2} \otimes \cdots  \otimes C^{j_n}$};
		\node (C) at (0,-2) {$C^{i_1}\otimes C^{i_2}\otimes \cdots \otimes C^{i_n}$};
		\draw[->] (A) --node [above] {$\Delta^{n-1}$} (B);
		\draw[->] (B) --node [ below right] {$\operatorname{pr}_{i_1,..,i_n}$} (C);	
		\draw[->] (A) --node [left] {$\Delta_{i_1,..,i_n}$} (C)	;
		
	\end{tikzpicture}
\end{center}
This framework naturally leads us to consider the maximal coideal in a graded coalgebra. For notational convenience, we write $\Delta_{1,1,...,1}$ by $\Delta_{1^n}$.
Let \( I_C(n) = \ker(\Delta_{1^n}) \) for all \( n \geq 1 \), and
\[
I_C = \bigoplus_{n \geq 1} I_C(n) = \bigoplus_{n \geq 2} I_C(n).
\]
Note that \( I_C(1) = 0 \) since \( \Delta_1 = \text{id} \). \par 
\begin{lemma}\textup{[\citealp{rootsys}, Lemma 1.3.13, Proposition 1.3.14]\label{C-lem 1.1}}\\
\text{$(1)$} Assume that $C$ is connected, that is, $C^0$ is one dimensional. Then $I_C \subseteq C$ is a coideal of $C$.\\ 
\text{$(2)$} 
    \text{Assume that \( C \) is connected. Then the following are equivalent.}\par 
        \text{$(a)$} \( C \) is strictly graded, that is, $P(C)=C^1$.\par 
    \text{$(b)$}  For all \( n \geq 2 \), \( \Delta_{1^n}: C^n \to (C^{1})^{\otimes n} \) is injective.\par 
       \text{$(c)$} For all \( n \geq 2 \), \( \Delta_{1,n-1}: C^n  \to C^1  \otimes C^{n-1} \) is injective.\par 
\text{$(d)$} \( I_C = 0 \).
\end{lemma}
The following technical result provides the necessary machinery to analyze the semisimplicity properties of $L$.
\begin{lemma}\label{C-lem3.4}\textup{[\citealp{rootsys},Proposition 13.2.3}]
Let $C$ be an $\mathbb{N}_0$-graded coalgebra, $X$ an $\mathbb{N}_0$-graded left $C$-comodule, and $Y$ a $C$-subcomodule of $X$. Let $k > 0$ be an integer. Assume that the map $\delta_{n-k,k}^X: X(n) \to C(n-k) \otimes X(k)$ is injective for all $n\geq  k$, and $Y$ is not contained in $\bigoplus_{i=0}^{k-1} X(i)$.
Then $Y \cap \bigoplus_{i=0}^k X(i) \neq 0$.
\end{lemma}

\begin{lemma}\label{C-lem3.5}
    If $L' \subseteq L$ is a non-zero $\AN$-subcomodule, then $L' \cap M \neq 0$.
\end{lemma}
\begin{proof}
We begin by observing that $\mathcal{B}(M \oplus N)$ carries a natural left $\mathcal{B}(N)$-comodule structure in the category $\GG$, defined via the composition: $$ \delta: \bB(M\oplus N)\xrightarrow {\Delta_{\bB(M\oplus N)}}\bB(M\oplus N) \otimes \bB(M\oplus N)\xrightarrow {\pi_{\BN} \otimes \operatorname{id}} \bB(N) \otimes \bB(M\oplus N).$$
    Now let $X \in L$, then $\delta(X) \in \BN \otimes L$ by Lemma \ref{C-lem3.3}, which implies that L inherits an  $\mathbb{N}_0$-graded  $\BN$-comodule  structure via the restriction of
    $\delta: L \rightarrow \BN \otimes L$. \par 
    Suppose $X \in L(n)$ for $n \geq 1$, consider the component map $\delta_{n-1,1}: L(n)\rightarrow \BN(n-1)\otimes L(1)=\BN(n-1)\otimes M.$ By Lemma \ref{C-lem3.3}, $\Delta_{\bB(M\oplus N)}(X)=X\otimes 1 + \BN \otimes L$,  it follows that
    $$ \delta_{n-1,1}(X)=(\Delta_{\bB(M\oplus N)})_{n-1,1}(X).$$
    As Nichols algebras are strictly graded  connected coalgebras, the map $(\Delta_{\mathcal{B}(M \oplus N)})_{n-1,1}$ is injective, and hence so is $\delta_{n-1,1}$. We now apply Lemma \ref{C-lem3.4} with $C=\BNG$, $X=L$, $Y=L'$ and $k=1$. Since $L(0)=0$ by its gradation. Therefore, we conclude that $L' \cap L(1)=L' \cap M\neq 0$, which completes the proof.
\end{proof}
Now we are ready to show the semisimplicity of $L$.

\begin{lemma}\label{C-irrlem}
     \text{$(1)$} Assume $M=\bigoplus_{i \in I} M_i$ is a direct sum  in $\GG$. Let $L_i=\operatorname{ad}\bB(N)(M_i)$ for all $i \in I$. Then we have such a decomposition in $\BNG$
    $$ L=\bigoplus_{i \in I} L_i.$$
  \text{$(2)$} If $M$ is irreducible in $\GG$, then $L=\operatorname{ad}\left(\BN\right)(M)$ is irreducible in  $\BNG$. 
\end{lemma}
\begin{proof}
 (1)  It suffices to show that the sum is direct. Suppose, for contradiction, that the sum is not a direct sum. Then there exists an index $k \in I$ such that
 $L_k \cap \sum_{i \neq k}L_i \neq 0$. By Lemma \ref{C-lem3.5}, this implies $$(L_k \cap \sum_{i \neq k}L_i) \cap M_k\neq 0,$$ hence $\sum_{i \neq k}L_i \cap M_k \neq 0 $. But since $M_k$ lies in degree one in $L_k$, we obtain 
$$\sum_{i \neq k} M_i \cap M_k \neq 0,$$
contradicting the directness of the sum $M = \bigoplus_{i \in I} M_i.$
    \par 
    (2) Now let $0\neq L' \subseteq L$ be a subobject in $\BNG$, then $L' \cap M \neq 0$ by Lemma \ref{C-lem3.5}.  Furthermore, $L' \cap M$ is a subobject of $M$ in $\GG$. But $M $ is irreducible in $\GG$, then $L'\cap M=M$. Hence 
    $$ L=\operatorname{ad}(\AN)(M)=\operatorname{ad}(\AN)(L'\cap M) \subseteq L',$$
    which implies $L^{\prime} = L$. We conclude that $L$ is irreducible.
\end{proof}
\subsection{Realizing $K$ as a Nichols algebra}
We now proceed to prove that $K$ possesses the structure of a Nichols algebra in an appropriate Yetter-Drinfeld module category. 
\begin{thm}\label{C-thm3.7}
    There is an isomorphism 
    \begin{equation}
        K \cong \bB(L)
    \end{equation}
    of Hopf algebras in the category $\BNG$. In particular, $P(K)=L$.
\end{thm}
To prepare for the proof of this theorem, we first establish a generation property:
\begin{lemma}\label{C-lem3.8}
    The algebra $K$ is generated by $L$ in $\BNG$.
\end{lemma}
\begin{proof}
    We have previously established that $L \subseteq K$ is a subobject of $K$ in $\BNG$. Let  $K'$ be the subalgebra of  $K$ generated by $L$ in $\BNG$. Obviously, $K'$ inherits the structure of an object in  $\BNG$ with structures induced by tensor product of $L$. \par Recall from Lemma \ref{C-lem3.2} that we have an algebra isomorphism  $K \# \BN \cong \bB(M\oplus N)$ in $\GG$ where the multiplication is induced by that of  $K \#(\BN \#\kG)$. Now we let $W$ be the image of $K' \# \BN$ under the isomorphism $K \# \BN \cong \bB(M\oplus N)$. To prove the lemma, it suffices to show that $W=\bB(M\oplus N)$. Note that $M \oplus N \subseteq W$ since $M,N \subseteq W$. Thus, we need only verify that $W$ forms a subalgebra of $\bB(M\oplus N)$. \par 
    A crucial observation is that $K'$ remains stable under  the  operation induced by $\BN$.
    Indeed, $L$ itself is stable under the  operation of $\BN$. Now consider arbitrary homogeneous elements $Y,Z\in L$ and $A \in \AN$, by (\ref{C-1.12}),
    \begin{align*}
        &\operatorname{ad}(A)(Y\otimes Z)\\&=\frac{\Phi(A_1,Y_{-1},Z_{-2})\Phi((\operatorname{ad}(A_2)(Y_0))_{-1},(\operatorname{ad}(A_4)(Z_0))_{-1},A_5))}{\Phi((\operatorname{ad(A_2)(Y_0)})_{-2},A_3,Z_{-1})}(\operatorname{ad}(A_2)(Y_0))_0\otimes(\operatorname{ad}(A_4)(Z_0))_0\\
&=\frac{\Phi(A_1,Y_{-1},Z_{-2})\Phi((\operatorname{deg}(\operatorname{ad}(A_2)(Y_0)),\operatorname{deg}(\operatorname{ad}(A_4)(Z_0)),A_5))}{\Phi((\operatorname{deg}(\operatorname{ad}(A_2)(Y_0)),A_3,Z_{-1})})\operatorname{ad}(A_2)(Y_0)\otimes \operatorname{ad}(A_4)(Z_0).
    \end{align*} 
    Since the associator vanishes if any  component does not lie in $G$.
    Note that $Y_0,Z_0 \in L$, it follows that $\operatorname{ad}(A)(Y\otimes Z)\in L \otimes L$. It follows that $K'$ remains stable under  the  operation induced by $\BN$ by induction.

    As both $K'$ and $\mathcal{B}(N)$ are subalgebras of $\mathcal{B}(M \oplus N)$, we only need to examine the case when $A \in N$ and $X \in K'$ are homogeneous. In this situation:
    \begin{align*}
        \operatorname{ad}(A)(X)=AX-\operatorname{ad}(a)(X)A.
    \end{align*}
    Since $\operatorname{ad}(A)(X)\in K'$ and $\operatorname{ad}(a)(X)A\in K' \# N$, we conclude that  $AX \in W$. Thus $W$ is a subalgebra of $\bB(M\oplus N)$, which establishes that $K$ is generated by $L$ in $\BNG$.
\end{proof}
To complete our analysis, we require an additional result concerning coalgebra filtrations. 
Let $C$ be a coalgebra. Recall that an $\mathbb{N}_0$-filtration  of $C$ is a family of subspaces $C_n$, $n \geq 0$, of $C$ satisfying
\begin{itemize}
    \item $C_n$ is a subspace of $C_m$ for all $m, n \in \mathbb{N}_0$ with $n \leq m$ 
    \item $C = \bigcup_{n \in \mathbb{N}_0} C_n$, 
    \item $\Delta_C(x) \in \sum_{i=0}^n C_i \otimes C_{n-i}$ for all $x \in C_n$, $n \in \mathbb{N}_0$.
\end{itemize}

\begin{lemma}\label{C-lem3.11}
Let $C$ be a coalgebra having an $\mathbb{N}_0$-filtration $\{C_n\}_{n\geq 0}$. Let $U$ be a non-zero comodule of $C$. Then there exists $u \in U \setminus \{0\}$ such that $\delta(u) \in C_0 \otimes U$.
\end{lemma}

\begin{proof}
Let $x \in U \setminus \{0\}$. Then there exists $n \in \mathbb{N}_0$ with $\delta(x) \in C_n \otimes U$. If $n = 0$, we are done. Assume now that $n \geq 1$ and let $\pi_0: C \to C/C_0$ be the canonical linear map. Since $C = \bigcup_{n \in \mathbb{N}_0} C_n$ is a coalgebra filtration, there is a maximal $m \in \mathbb{N}_0$ such that
\[
\pi_0(x_{-m}) \otimes \cdots \otimes \pi_0(x_{-1}) \otimes x_{0} \neq 0.
\]
Let $f_1, \ldots, f_m \in C^*$ with $f_i|_{C_0} = 0$ for all $i \in \{1, \ldots, m\}$ such that
\[
y := f_1(x_{-m}) \cdots f_m(x_{-1})x_{0} \neq 0.
\]
Then $\delta(y) = f_1(x_{-m-1}) \cdots f_m(x_{-2})x_{-1} \otimes x_{0} \in C_0 \otimes U$ by the maximality of $m$. 
\end{proof}
\begin{proof}[Proof of Theorem \ref{C-thm3.7}]
    We endow a non-standard grading on $\bB(M\oplus N)$ via setting $\operatorname{deg}(M)=1$ and $\operatorname{deg}(N)=0$. Under this grading, $\bB(M\oplus N)$ remains  a $\mathbb{N}_0 $-graded Hopf algebra in $\GG$. Furthermore, $\bB(M\oplus N) \# \kG$ becomes a $\mathbb{N}_0$-graded coquasi-Hopf algebra with $\operatorname{deg}(\kG)=0$. It follows direct that $K$ inherits an  $\mathbb{N}_0$-graded Hopf algebra in $\BNG$ with $K(n)=K \cap (\bB(M\oplus N) \# \kG)(n)$.  In particular, we have $K(1)=K \cap (\bB(M\oplus N) \# \kG)(1)=K \cap (\bB(M\oplus N)(1))=\operatorname{ad}(\BN)(M)=L $. Hence $K$ is generated by its degree one component $K(1)$.\par 
    It remains to  show  that $P(K)=K(1)$. That is, there is no primitive element in $K(n)$, $n\geq 2$. Suppose, for contradiction, that there exists a nonzero subspace $U \subseteq P(K(n))$ for some positive integer $n$. By Lemma \ref{C-lem3.9}, $U$ is an object in $\BNG$, since $\operatorname{deg}(N)=\operatorname{deg}(\kG)=0$. Now, consider the coalgebra filtration on $\BN\# \kG$ given by
    $(\BN\#\kG)_0=\kG$, $(\BN\#\kG)_1=\kG+N$. Applying Lemma \ref{C-lem3.11}, we find a nonzero element $u \in U$ such that $\delta_K(u)=u_{-1}\otimes u_0\in \kG\otimes U$. Note that 
    $$ \delta_K(u)=(\pi\otimes \operatorname{id})\Delta_{\AMN}(u).$$ Therefore
     $$\Delta_{\AMN}(u)=u \otimes 1+u_{-1}\otimes u_0.$$
    Applying the map $\mathrm{id} \# \varepsilon \otimes \mathrm{id}$, we obtain $$\Delta_{\bB(M\oplus N)}(u)=(\operatorname{id}\# \varepsilon \otimes \operatorname{id})\Delta_{\AMN}(u)=1\otimes u+u\otimes 1,$$
      which implies that $u$ is a primitive element in $\mathcal{B}(M \oplus N)$. However, since $K(n)$ is generated by $L = \operatorname{ad}(\mathcal{B}(N))(M)$, the element $u$ must have degree at least $n$ in the standard grading of $\mathcal{B}(M \oplus N)$. This leads to a contradiction, as primitive elements in a Nichols algebra lie in degree one.
\end{proof}
\subsection{From the Nichols algebra back to the space of coinvariants}
We now turn to a converse of the preceding result under an additional restriction.\par 
Let $C$ be a coalgebra and $D \subseteq C$ be a subcoalgebra. Let $V$ be a left comodule of $C$ with  comodule structure $\delta: V \rightarrow C\otimes V$.  We denote the largest $D$-subcomodule of $V$ by 
$$ V(D)=\{ v\in V\mid \delta(v)\in D \otimes V\}. 
$$

\begin{lemma}\label{C-lem 3.10}
Let $N \in \GG$ and $W \in \BNG$. Assume that $W=\bigoplus_{i \in I}W_i$ is a decomposition into irreducible objects in $\BNG(\Zgr)$. Let $M=W(\kG)$, and $M_i=M \cap W_i$ for all $i \in I$. \par
\text{$(1)$} $M=\bigoplus_{i\in I}M_i$ is a decomposition into irreducible
objects in $\GG(\Zgr)$.\par 
\text{$(2)$} For all $i \in I$, $M_i$ is the $\mathbb{Z}$-homogeneous component of $W_i$ of minimal degree, and $W_i=\BN\rhd M_i=\bigoplus_{n\geq 0}N^{\otimes n}\rhd M_i$.

\end{lemma}

\begin{proof}
    Let $W=\bigoplus_{n\in \mathbb{Z}}W(n)$ be the $\mathbb{Z}$-grading of $W$ in $\BNG$. Then $M$ is a $\mathbb{Z}$-graded object in $\GG$ with homogeneous components $M(n)=M \cap W(n)$ for all $n \in \mathbb{Z}$. It is clear that $M=\bigoplus_{n\in \mathbb{Z}}\bigoplus_{i \in I}M_i(n)=\bigoplus_{i \in I}M_i$, where $M_i=M \cap W_i=W_i(\kG)$.\par
    We now show that each  $M_i$ is an  irreducible object in $\GG(\Zgr)$ for each $i \in I$. First, by Lemma \ref{C-lem3.11}, 
  we have $M_i \neq 0$. Suppose $0 \neq M_i' \subseteq M_i$ is a homogeneous subobject in $\GG$, and let $m_i$ be the degree of $M_i'$. Define
 $$W_i':=\BN \rhd M_i'.$$
Using a method similar to that in Lemma \ref{C-lem3.2}, one verifies that $W_i'$ is a $\mathbb{Z}$-graded $\mathcal{A}(N)$-subcomodule. Moreover,
  $\AN \rhd (\BN \rhd M_i') \subseteq \AN \rhd M_i'=\BN \rhd M_i'$. Therefore $W_i'$ is a $\mathbb{Z}$-graded subobject of $W_i$ in $\BNG$.
The object $W_i'$ is concentrated in degrees $\geq m_i$, and its degree $m_i$'s component is exactly $M_i'$. Since $W_i$ is irreducible, we must have $W_i' = W_i$. Therefore,
$M_i'=W_i'\cap M=W_i \cap M=M_i$, which establishes the irreducibility of $ M_i$. Moreover, $M_i$ is indeed the homogeneous component of $W_i$ of minimal degree. Finally, for each $i \in I$ and $n \in \mathbb{N}_0$, 
    $$ \operatorname{deg}(N^{\otimes n} \rhd M_i)=n+\operatorname{deg}(M_i),$$ since the  map $\BN \# \kG \otimes W_i \rightarrow W_i$ is $\mathbb{Z}$-graded. It follows that $\BN \rhd M_i=\bigoplus\limits_{n \geq 0}N^{\otimes n} \rhd M_i$ as required.
\end{proof}
\begin{thm}\label{C-prop3.11}
    Let $N \in \GG$ and  $W$ be a semisimple object in the category $\BNG(\Zgr)$. Here $\AN$ is equipped with the standard $\mathbb{N}_0$-grading. Let $K=\mathcal{B}(W)$ be the Nichols algebra of $W$ in $\BNG$ and define $M=W(\kG)$. Then there is a unique isomorphism 
    \begin{equation}
        K \# \BN\cong \mathcal{B}(M\oplus N)
    \end{equation} of braided Hopf algebras in $\GG$ which is the identity on $M \oplus N$.
    \end{thm}
\begin{proof}
  Let $W=\bigoplus_{i \in I} W_i$ be the decomposition of $W$ into irreducible objects in the category \\$\BNG(\Zgr)$. For each $i \in I$, define $M_i=M\cap W_i$. By Lemma \ref{C-lem 3.10}, $W_i=\bigoplus_{n \in I}N^{\otimes n}\rhd M_i$. Let $\operatorname{deg}(M_i)=m_i$ for each $i \in I$. We endow $W_i$ a new grading by $\widetilde{W_i}=W_i$ with  
    $$ \widetilde{W_i}(n)=W(n+m_i-1)=N^{n-1}\rhd M_i,\ \text{for all} \ n \in \mathbb{N}_0.$$
 With the new grading, $\widetilde{W}=\bigoplus_{i \in I}\widetilde{W_i}$ remains an object in $\BNG$ since degree shifting preserves graded modules and graded comodules. Moreover, $\widetilde{W}(n)=0$ for all $n \leq 0$.  
So the Nichols algebra $\bB(\widetilde{W})$  is an $\mathbb{N}_0$-graded Hopf algebra in $\BNG$.  Note that $\bB(\widetilde{W})\cong \bB(W)$, as both are determined by the same module and comodule maps.
Under this grading, $\widetilde{W}(0)=0$ and $\widetilde{W}(1)=M$, the degree zero and degree one components of the Nichols algebra $\bB(\widetilde{W})$ are   $$\bB(\widetilde{W})(0)=\mathbbm{k}, \ \ \bB(\widetilde{W})(1)=M.$$ Moreover, $K\#(\BN \# \kG)$ is a $\mathbb{N}_0$-graded coquasi-Hopf algebra, and 
$$ R:=K\#\BN=(K \#(\BN\# \kG))^{\operatorname{co}\kG}$$
is a $\mathbb{N}_0$-graded Hopf algebra in $\GG$. Its degree $0$ component is $\mathbbm{k}$, and its degree $1$ component is 
$M\oplus N$. \par 
We now show that 
$R$ is generated by  $M \oplus N$ as an algebra in $\GG$. It is direct to see that $R$ is generated by $K(1)=\BN \rhd M$ and $N$. It therefore suffices to prove that $\BN \rhd M=\bigoplus_{n \geq 0}N^{\otimes n} \rhd M$ is contained in the subalgebra generated by $\BN$ and $M$. To see this, we proceed by induction on 
 $n$. For the base case, take homogeneous elements $X \in M$ and $ Y \in N$. Inside the coquasi-Hopf algebra $K\#(\BN \# \kG)$, we have
$$ YX=(Y_1\rhd X)Y_2=(y \rhd X)Y+(Y \rhd X).$$
Since $M \in \GG$ and $y \in G$, the term $Y \rhd X$ is contained in the subalgebra generated by $\BN$ and $M$. Now assume for some $k\geq 0$, that for all  homogeneous $Z \in \BN(k)$ and $X \in M$, the element $Z \rhd X$
lies in the subalgebra generated by $\BN$ and $M$, then for homogeneous $Y \in N$,
\begin{align*}
    (YZ)\rhd X&=\frac{\Phi(y,(Z_2\rhd X)_{-1},Z_3))}{\Phi(y,Z_1,x)}Y\rhd(Z_2\rhd X)_0\\
    &=\frac{\Phi(y,(Z_2\rhd X)_{-1},Z_3))}{\Phi(y,Z_1,x)}(Y(Z_2\rhd X)_0-(y\rhd(Z_2\rhd X)_0 ))\\
    &=\frac{\Phi(y,\operatorname{deg}(Z_2\rhd X),Z_3))}{\Phi(y,Z_1,x)}(Y(Z_2\rhd X)-y\rhd(Z_2\rhd X)).
\end{align*} 
By the induction hypothesis, the first term lies in the subalgebra generated by
 $\BN$ and $M$. Since $y$ stabilizes both $\BN$  and $M$, the second term also lies in that subalgebra. This completes the induction. Therefore, $R$ is generated by  $M \oplus N$, and hence a pre-Nichols algebra of $M\oplus N$.\par 
By the universal property of $\bB(M\oplus N)$, there exists a surjective morphism of $\mathbb{N}_0$-graded Hopf algebras in $\GG$:
$$ \rho : R \rightarrow \bB(M \oplus N), \rho\mid_{M\oplus N}=\operatorname{id}.$$
This induces a surjective coquasi-Hopf algebra map:
$$ \rho \# \operatorname{id}:R \# \kG \rightarrow \bB(M \oplus N)\# \kG.$$
Let $K'=(\bB(M\oplus N)\# \kG)^{\operatorname{co}\BN\#\kG}$. Then we obtain two bijective coquasi-Hopf algebra maps:
$$ R\# \kG \rightarrow K \# (\BN \# \kG), \ \ K' \# (\BN \# \kG)\rightarrow \bB(M\oplus N)\# \kG.$$
Then the map $\rho \# \operatorname{id}$ thus induces a surjective map of coquasi-Hopf algebras
$$ \rho': K\#(\BN \#\kG)\rightarrow K'\#(\BN \# \kG), \ \ \rho'\mid_{(M \oplus N)}=\operatorname{id}.$$
Moreover,  the following diagram commutes:
 \begin{center}
	\begin{tikzpicture}
		\node (A) at (0,0) {$K \# (\BN \# \kG)$};
		\node (B) at (6,0) {$K'\#(\BN\#\kG)$};
		\node (C) at (0,-3) {$\BN \# \kG$};
  	\node (D) at (6,-3) {$\BN \# \kG$};
		\draw[->] (A) --node [above] {$\rho'$} (B);
		\draw[->] (A) --node [  right] {$\varepsilon_K \# \operatorname{id}_{\BN \# \kG}$} (C);	
		\draw[->] (B) --node [right] {$\varepsilon_{K'} \# \operatorname{id}_{\BN \# \kG} $} (D)	;
		\draw[->] (C) --node [above] {$\operatorname{id} $} (D)	;
	\end{tikzpicture}
\end{center}
since ${\rho\mid}_{M\oplus N}=\operatorname{id}$. As $\BN \# \kG$ acts on $K'$ via adjoint operation.  Theorem $\ref{C-thm3.7}$ implies that  $K'\cong \bB(\operatorname{ad}(\BN)(M))$.
Therefore $\rho'$ induces a surjective map 
$$\phi: K\rightarrow \bB(\operatorname{ad}(\BN)(M))$$
in $\BNG$ between the right coinvariant subspaces of $\varepsilon_K \# \operatorname{id}_{\BN \# \kG}$ and $\varepsilon_{K'} \# \operatorname{id}_{\BN \# \kG} $, satisfying ${\phi\mid_M}=\operatorname{id}$.
Furthermore, there is a surjective map in $\BNG$:
$$ \phi_1: \BN \rhd M \rightarrow \operatorname{ad}(\BN)(M), \ \ \phi_1\mid_M=\operatorname{id}.$$
Since  $M=\bigoplus_{i\in I}M_i$ is a decomposition into irreducible
objects in $\GG$, each $\operatorname{ad}(\BN)(M_i)$ is irreducible in $\BNG$. On the other hand, each $\BN \rhd M_i$ is irreducible as well. Thus $\phi_1$ is an isomorphism in $\BNG$, and it follows that  $\phi$ is an isomorphism of Hopf algebra in $\BNG$. Consequently,
$\rho \# \operatorname{id}_{\mathbbm{k}G}=\phi\#\operatorname{id}_{\AN}$ is an isomorphism of coquasi-Hopf algebras. 
Therefore,
$$ \rho=(\rho \#\operatorname{id}_{\mathbbm{k}G})^{\operatorname{co}\mathbbm{k}G} : K\# \BN \rightarrow \bB(M\oplus N)$$ is an isomorphism of Hopf algebras in $\GG$.

\end{proof}
\section{Reflection of Nichols algebras}
In this section, we develop  reflections of Nichols algebras over pointed cosemisimple coquasi-Hopf algebras, which will serve as the cornerstone for our subsequent construction of semi-Cartan graphs. The reflection  allows us to systematically relate different Nichols algebra realizations through a well-defined transformation procedure.
\par 
\subsection{Definition of reflection and basic properties}
Let us begin by establishing the basic framework. We still assume $G$ is a finite group, $\Phi$ is a $3$-cocycle on $G$.  Fix a positive integer $\theta$ and  denote the index set $\mathbb{I}=\{ 1,2,...,\theta\}$. 
\begin{definition}\label{C-def5.1}
    Let $\mathcal{F}_{\theta}$ denote the class of all $\theta$-tuples $M = (M_1, \ldots, M_{\theta})$, where $M_1, \ldots, M_{\theta} \in \GG$ are finite-dimensional  Yetter-Drinfeld modules. If $M \in \mathcal{F}_{\theta}$, we define

\[
\mathcal{B}(M): = \mathcal{B}(M_1 \oplus \cdots \oplus M_{\theta}).
\]

For tuples $M, M' \in \mathcal{F}_{\theta}$, we write $M \cong M'$, if $M_j \cong M_j'$ in $\GG$ for each $j$.
The isomorphism class of $M \in \mathcal{F}_{\theta}$ is denote by $[M]$.
\par 
For $1 \leq i \leq \theta$ and $M \in \mathcal{F}_{\theta}$, we say the tuple $M$ admits the $i$-th reflection $R_i(M)$  if for all $j \neq i$ there is a natural number $m_{ij}^M \geq 0$ such that $(\mathrm{ad}\, M_i)^{m_{ij}^M} (M_j)$ is a non-zero finite-dimensional subspace of $\mathcal{B}(M)$, and $(\mathrm{ad}\, M_i)^{m_{ij}^M + 1} (M_j) = 0$. Assume  $M$ admits the $i$-th reflection. Then we set $R_i(M) = (V_1, \ldots, V_{\theta})$, where

\[
V_j = 
\begin{cases} 
M_i^*, & \text{if } j = i, \\
(\mathrm{ad}\, M_i)^{m_{ij}^M} (M_j), & \text{if } j \neq i.
\end{cases}
\]
\end{definition}
Having defined individual reflections, we now extend this notion to sequences of reflections, which will be important for our study of repeated reflections of  tuples.
\begin{definition}\label{C-def5.2}
    Suppose $M \in \mathcal{F}_{\theta}$, such that $M_i$ is irreducible for each $i$. Let $l \in \mathbb{N}_0$ and $i_1,i_2,...,i_l\in \mathbb{I}$, \par 
   \text{$(1)$} We say $M$ admits the reflection sequence $(i_1,i_2,...,i_l)$ if $l=0$ or $M$ amdits the $i_1$-th reflection and  $R_{i_1}(M)$ satisfies the reflection sequence  $(i_2,i_3,...,i_l)$.\par
    \text{$(2)$} We say $M$ admits all reflection sequence if $M$ admits all reflection sequence $(i_1,i_2,...,i_l)$ for all $l \in \mathbb{N}_0$ and $i_1,i_2,...,i_l \in \mathbb{I}$.
\end{definition}
For $M \in \mathcal{F}_{\theta}$ admitting all reflections, we denote 
$$\mathcal{F}_{\theta}(M)=\{R_{i_1}(...(R_{i_l}(M))...)\mid l\in \mathbb{N}_0,\ i_1,..,i_l \in \mathbb{I}\}.    $$
\par A natural question arises: how do the irreducibility properties behave under reflections? The following result provides a reassuring answer.
\begin{lemma}\label{C-lem3.13}
    Suppose $M \in \mathcal{F}_{\theta}$ admits the $i$-th reflection for some $i \in \mathbb{I}$, and $M_j$ is irreducible for each $j \in \mathbb{I}$. Then each $R_i(M)_j$ is irreducible in $\GG$ for  $1 \leq j \leq \theta$.
\end{lemma}
\begin{proof}
    By definition,  $R_i(M)$ is defined, thus ${R_i(M)}_i \cong M_i^*$ is irreducible since $M_i$ is. For $j \neq i$, we observe that
    $\operatorname{ad}\bB(M_i)(M_j)=\bigoplus_{n=0}^{m_{ij}}\operatorname{ad}(M_i)^n(M_j)$. Consider the embedding
    $\mathbbm{k}G \rightarrow \mathcal{A}(M_i) $ and the projection $\mathcal{A}(M_i) \rightarrow \kG$. The object 
    $\operatorname{ad}(M_i)^{m_{ij}}(M_j)$ belongs to $\GG$. Now suppose $0 \neq L \subset \operatorname{ad}(M_i)^{m_{ij}}(M_j)$ is a Yetter-Drinfeld submodule over $\GG$. 
    Let $\langle L \rangle$ be the $\mathcal{A}(M_i)$-subcomodule of $\operatorname{ad}\bB(M_i)(M_j)$ generated by $L$,  defined explicitly as
    $$ \langle L \rangle:=\{\langle f,X_{-1}\rangle X_0\mid X \in L, f \in \operatorname{Hom}(\mathcal{A}(M_i),\mathbbm{k})\}.$$
    We want to prove 
    $ \langle L\rangle$ is a Yetter-Drinfeld submodule of  $\operatorname{ad}\bB(M_i)(M_j)$ in $\AMI$. To see this, take any homogeneous $A \in \mathcal{A}(M_i)$. Then  for all homogeneous $X \in L$, by Lemma \ref{C-lem2.6},
\begin{align*}
    \operatorname{ad}(A)&(\langle f,X_{-1}\rangle X_0)=\langle f,X_{-1}\rangle \operatorname{ad}(A)(X_0)\\
    &=\langle f,s(\mathcal{S}(A_1),(\operatorname{ad}(A_4)(X_0))_{-1}A_6 )t(A_3,X_{-1})\mathcal{S}(A_2)((\operatorname{ad}(A_4)(X_0))_{-2}A_5)\rangle (\operatorname{ad}(A_4)(X_0))_0.
\end{align*}
Note that $A$ and $X$ are homogeneous, thus if $(\operatorname{ad}(A_4)(X_0))_{-1} \notin G$, then this term equals zero. 
Therefore \begin{align*}
   & \operatorname{ad}(A)(\langle f,X_{-1}\rangle X_0)\\&=\langle f,s(\mathcal{S}(A_1),\operatorname{deg}(\operatorname{ad}(A_4)(X_0))A_6 )t(A_3,X_{-1})\mathcal{S}(A_2)((\operatorname{ad}(A_4)(X_0))_{-1}A_5)\rangle (\operatorname{ad}(A_4)(X_0))_0.
\end{align*}
Hence $\operatorname{ad}(A)(\langle f,X_{-1}\rangle X_0) \in \langle L\rangle.$
This ensures the first axiom of a Yetter-Drinfeld module, and the third axiom holds automatically.
  Therefore $ \langle P\rangle$ is a Yetter-Drinfeld submodule of  $\operatorname{ad}\bB(M_i)(M_j)$ in $\AMI$. \par 
    Since $\operatorname{ad}(\bB(M_i))(M_j)$  is irreducible in 
$\AMI$ by Lemma \ref{C-irrlem} (2), and $P \neq 0$, we have $\langle P \rangle= \operatorname{ad}(\bB(M_i))(M_j)$. It is direct to see that $ \langle P \rangle= \bigoplus_{i=0}^{m_{ij}} \langle P \rangle\cap \bB^{i}(M_i \oplus M_j)$. Then $$P=\langle P \rangle\cap \bB^{m_{ij}}(M_i \oplus M_j)=\operatorname{ad}(\bB(M_i))(M_j)\cap \bB^{m_{ij}}(M_i \oplus M_j)=\operatorname{ad}(M_i)^{m_{ij}}(M_j).$$ This establishes the irreducibility.
\end{proof}
\begin{lemma}\label{C-lem5.4}
Suppose $M \in \mathcal{F}_{\theta}$ and $M$ admits the  $i$-th reflection for each $i \in \mathbb{I}$.  We define $a_{ii}^M = 2$ for all $1\leq i \leq \theta$ and define $a_{ij}^M = -m_{ij}^M$. Then $(a_{ij}^M)_{i,j\in \mathbb{I}}$  is a generalized Cartan matrix.
\end{lemma}
\begin{proof}
    To establish this result, we need only verify: if $1\leq i<j \leq \theta $ such that $a_{ij}^M=0$, then $a_{ji}^M=0$.\par 
    Let $X \in M_i, Y \in M_j$ be homogeneous. A direct computation yields:
   \begin{align*}
       \Delta_{\mathcal{B}(M)}(\operatorname{ad}(X)(Y))&=\Delta_{\mathcal{B}(M)}(XY-(\operatorname{ad}(x)(Y))X)\\
       &=(1\otimes X+ X\otimes 1)(1 \otimes Y+ Y \otimes 1)-( \operatorname{ad}(x)(Y)\otimes 1+1 \otimes \operatorname{ad}(x)(Y))(1\otimes X+X \otimes 1)\\
       &=XY\otimes 1+X\otimes Y+1\otimes XY+\operatorname{ad}(x)(Y)\otimes X\\&-\operatorname{ad}(x)(Y)\otimes X-\operatorname{ad}(x)(Y)X \otimes 1-1\otimes \operatorname{ad}(x)(Y)X-\operatorname{ad}(xyx^{-1})(X)\otimes \operatorname{ad}(x)(Y) \\
      & =\operatorname{ad}(X)(Y) \otimes 1+ 1\otimes \operatorname{ad}(X)(Y)+X\otimes Y-c^2(X\otimes Y).
   \end{align*} 
Note that   $a_{ij}^M=0$ implies $\operatorname{ad}(X)(Y)=0$, from which we deduce $X\otimes Y-c^2(X\otimes Y)=0$.
This establishes $(\operatorname{id}-c^2)(M_i \otimes M_j)=0$. Furthermore, we have $(\operatorname{id}-c^2)c(M_j \otimes M_i)=0$. Since the braiding 
$c$ is invertible, it follows that  $(\operatorname{id}-c^2)(M_j \otimes M_i)=0.$ Consequently, $\Delta_{\mathcal{B}(M)}(\operatorname{ad}(Y)(X))=0$ in $\bB(M_i \oplus M_j)$, for all $X \in M_i,  Y \in M_j$, which implies $\operatorname{ad}(Y)(X)=0$.  We conclude that  $a_{ji}^M=0$, as required.
\end{proof}
\begin{lemma}\label{C-lem3.15}
    Suppose $M \cong N$ in $\mathcal{F}_{\theta}$, if  $M$  admits the $i$-th reflection for some $i \in \mathbb{I}$, so does $N$. Furthermore, $R_i(M)\cong R_i(N)$ and $a_{ij}^M=a_{ij}^N $ for each $j \in \mathbb{I}$.
\end{lemma}
\begin{proof}
    Let $\phi: M \rightarrow N$ be an  isomorphism of tuples with each component 
 $\phi_j:M_j \rightarrow N_j$ an isomorphism of Yetter-Drinfeld modules.
We first demonstrate that 
$\phi_i$ inducing a unique morphism $\mathcal{B}(\phi_i):\bB(M_i)\rightarrow \bB(N_i)$. \par The tensor algebra map $T(\phi_i): T(M_i) \rightarrow T(N_i) $ is a morphism of bialgebras in $\GG$. Since $T(\phi_i)$ is a $\mathbb{N}_0$-graded coalgebra map, it sends the coideal $I(M_i)=\bigoplus_{n\geq 2}\operatorname{ker}(\delta_{1^n})$ to $I(N_i)$. Hence $\phi_i$ descends to a well-defined morphism $\mathcal{B}(\phi_i):\mathcal{B}(M_i)\rightarrow \bB(N_i)$. 
This morphism is $\mathbb{N}_0$-graded, as its restriction to each homogeneous component $\mathcal{B}(V)(n)$ is induced by $\phi_i^{\otimes n}$. Explicitly, for any $n \in \mathbb{N}_0$ and $\gamma, \gamma_1, \ldots, \gamma_n \in \mathbb{N}_0$ with $\gamma = \gamma_1 + \cdots + \gamma_n$, we have
\[
\mathcal{B}(\phi_i)(V(\gamma_1) \cdots V(\gamma_n)) = \phi_i(V(\gamma_1)) \cdots \phi_i(V(\gamma_n)) \subseteq \mathcal{B}(W)(\gamma).
\] The claim on the surjectivity of $\mathcal{B}(\phi_i)$ is immediate. While injectivity follows from the equation 
\[
\Delta_{1^n} \phi_i^{\otimes n} = \phi_i^{\otimes n} \Delta_{1^n}
\]
for all $n \in \mathbb{N}_0$. We thus obtain an equivalence of categories: $\AMI \cong \ANI$.
For each $j \neq i \in \mathbb{I}$,
$$ \operatorname{ad}(\bB(\phi_i))(\phi_j):\operatorname{ad}(\bB(M_i))(M_j)\cong \operatorname{ad}(\bB(N_i))(N_j)$$ is an isomorphism of $\mathbb{N}_0$-graded objects in $\AMI$. In particular, $\operatorname{ad}(M_i)^{m_{ij}^M}(M_j)\cong \operatorname{ad}(N_i)^{m_{ij}^M}(N_j)\in \GG$ and $\operatorname{ad}(N_i)^{m_{ij}^M+1}(N_j)=0$. This establishes both $a_{ij}^M=a_{ij}^N$ and $R_i(M) \cong R_i(N)$.
\end{proof}
\subsection{Main Theorem}
In this subsection, we always fix a tuple $M = (M_1, \ldots, M_{\theta})\in \mathcal{F}_{\theta}$, with each component $M_i$ being irreducible, $i\in \mathbb{I}$. Furthermore, we assume $\bB(M_i)$ is finite-dimensional for each $i \in \mathbb{I}$.
 By Corollary \ref{C-cor3.8}, we have such a braided tensor equivalence for each $i\in \mathbb{I}$.
\begin{equation}
  \Omega_i:  \AMI \cong \AMID.
\end{equation}

The main result of this paper  proves that  the reflection of Nichols algebras in $\GG$ gives rise to a semi-Cartan graph. The first lemma provides some relations between reflection of $M$ and the braided tensor equivalence $\Omega_i$. 
\begin{lemma}
    With above assumptions on $M$, and suppose $M$ admits the $i$-th reflection with $R_i(M)=(V_1,...,V_{\theta})$. Then
    \begin{align*}
        \operatorname{ad}(\bB(M_i)(M_j))=\bigoplus_{n=0}^{m_{ij}}\operatorname{ad}(M_i)^n(M_j)&, \ \ \ \ V_j= \mathcal{F}_0(\operatorname{ad}\bB(M_i)(M_j)), \\
       \Omega_i(\operatorname{ad}(\bB(M_i)(M_j)))=\bigoplus_{n=0}^{m_{ij}}\operatorname{ad}(M_i^*)^n(V_j)&, \ \ \ \ M_j \cong \operatorname{ad}(M_i^*)^{m_{ij}}(V_j).
    \end{align*}
\end{lemma}
\begin{proof}
    The first equality follows directly from the decomposition:  $$\operatorname{ad}\left(\bB(M_i)(M_j)\right)=\bigoplus_{n \geq 0} \operatorname{ad}(M_i)^n(M_j).$$
    \par For the second statement, the inclusion $V_j \subseteq \mathcal{F}_0(\operatorname{ad}\bB(M_i)(M_j))$  is immediate. Suppose, for contradiction, that the containment is proper. Then there exists a nonzero $X \in \operatorname{ad}(M_i)^{l}(M_j)$ for $0 \leq l \leq m_{ij}-1$ such that  $ X \in \mathcal{F}_0(\operatorname{ad}\bB(M_i)(M_j))$. The Yetter-Drinfeld submodule generated by $X$ is contained in $\bigoplus_{k=0}^{l}\operatorname{ad}(M_i)^{k}(M_j)$. However, since  $V_j \neq 0$, and  $\operatorname{ad}(\bB(M_i)(M_j))$ is irreducible in $\AMI$ by Lemma \ref{C-irrlem}(ii), this leads to a contradiction. Hence, $V_j=\mathcal{F}_0(\operatorname{ad}\bB(M_i)(M_j))$.\par  
   For the third statement, we consider the category $\AMI(\Zgr)$ by setting $\operatorname{deg}(M_k)=1$ for $1\leq k \leq n$. Then $\operatorname{ad}(M_i)^l(M_j)$ has degree $l+1$ for $1\leq l \leq m_{ij}$. By Lemma \ref{C-lemOmega}, $\Omega_i(\operatorname{ad}(M_i)^l(M_j))=\operatorname{ad}(M_i)^l(M_j)$ is  $\mathbb{Z}$-graded with $\operatorname{deg}(\Omega_i(\operatorname{ad}(M_i)^l(M_j)))=-1-l$. Moreover,
    $$ \mathcal{F}^0\Omega_i(\operatorname{ad}\bB(M_i)(M_j))=\mathcal{F}_0(\operatorname{ad}\bB(M_i)(M_j))=\operatorname{ad}(M_i)^{m_{ij}}(M_j)=V_j.$$
    Since $\Omega_i(\operatorname{ad}(\bB(M_i)(M_j)))$ is irreducible in $\AMID$,  with the minimal degree part $V_j$. By Lemma \ref{C-lem 3.10},
    $$\Omega_i\left(\operatorname{ad}\left(\bB(M_i)(M_j)\right)\right)=\operatorname{ad}\left(\bB(M_i^*)(V_j)\right)=\bigoplus_{n=0}^{m_{ij}}\operatorname{ad}(M_i^*)^n(V_j).$$ In particular, we have $$M_j=\mathcal{F}^0\operatorname{ad}\bB(M_i)(M_j)=\mathcal{F}_0\operatorname{ad}\bB(M_i^*)(V_j)=\operatorname{ad}(M_i^*)^{m_{ij}}(V_j).$$
\end{proof}

The next theorem gives a natural explanation of reflections of tuples of Yetter-Drinfeld modules. Although the proof strategy parallels that of the Hopf algebra case [\citealp{rootsys}, Theorem 13.4.9], we include the details here for the sake of completeness and the reader's convenience.
\begin{thm}\label{C-thm 3.15}
With the above assumptions on $M$. For each $i\in \mathbb{I}$, suppose $M$ admits the $i$-th reflection, then there is an isomorphism of Hopf algebras in $\GG$:
     \begin{equation}
         \Theta_i:\bB(R_i(M))\cong \Omega_i\left(\bB(M)^{\operatorname{co}\bB(M_i)}\right)\# \bB(M_i^*).
     \end{equation}
\end{thm}
\begin{proof}
    Let $N=\bigoplus_{j\neq i}M_j$, and define $Q=\operatorname{ad}\bB(M_i)(N)$. By Lemma \ref{C-irrlem}, we have the decomposition $Q=\bigoplus_{j \neq i}Q_j$, where each $Q_j=\operatorname{ad}\bB(M_i)(M_j)$ is irreducible in $\AMI$.  Applying  Theorem \ref{C-thm3.7} yields an isomorphism of Hopf algebras:
    $$ \bB(M)^{\operatorname{co}\bB(M_i)} \cong \bB(Q).$$
    Now observe that the functor $\Omega_i$ sends Nichols algebras to Nichols algebras by Lemma \ref{C-lem2.18}. Therefore,
    $$\Omega_i\left(\bB(M)^{\operatorname{co}\bB(M_i)}\right)\cong \Omega_i(\bB(Q))\cong \bB(\Omega_i(Q)).$$
    Hence, in the category $\GG$,
    $$ \Omega_i\left(\bB(M)^{\operatorname{co}\bB(M_i)}\right) \# \bB(M_i^*)\cong \bB(\Omega_i(Q))\#\bB(M_i^*).$$
    Since $Q$ is a direct sum of irreducible objects, it is semisimple. The equivalence $\Omega_i$ preserves semisimplicity, so $\Omega_i(Q)$ is also semisimple in $\AMID(\Zgr)$.   By Proposition \ref{C-prop3.11}, we obtain
    $$ \bB\left(\Omega_i\left(Q\right)\right)\#\bB(M_i^*)\cong \bB\left(\mathcal{F}^0(\Omega_i(Q))\oplus M_i^*\right).$$
   We have already shown that $\mathcal{F}^0(\Omega_i(Q))=\bigoplus_{j \neq i}V_j$.
    Thus 
    $$ \bB(\Omega_i(Q))\#\bB(M_i^*)\cong \bB\left(\bigoplus_{j \neq i}V_j \oplus M_i^*\right)\cong \bB(R_i(M)).$$
   Combining the above isomorphisms, we conclude that in $\GG$, there is an isomorphism of Hopf algebras:
   $$  \Theta_i:\bB(R_i(M))\cong \Omega_i\left(\bB(M)^{\operatorname{co}\bB(M_i)}\right)\# \bB(M_i^*).$$
\end{proof}
\begin{cor}\label{C-cor5.8}
    With above assumptions on $M$. For each $i \in \mathbb{I}$, suppose $M$ admits $i$-th reflection, then $R_i(M) $  admits $i$-th reflection. Furthermore, we have:
    \begin{equation}
        R_i^2(M) \cong M,
    \end{equation}
    and 
    \begin{equation}
        a_{ij}^M=a_{ij}^{R_i(M)}, 
    \end{equation}
    for all $1 \leq j \leq \theta$.
\end{cor}
\begin{proof}
    By Theorem \ref{C-thm 3.15}, $R_i(M)=(V_1,...,M_i^*,...,V_{\theta}).$ 
    Note that 
    $$ \operatorname{ad}\bB(M_i^*)(V_j)=\Omega_i\left(\operatorname{ad}\bB(M_i)(M_j)\right).$$
It follows that 
    $$ \operatorname{ad}\bB(M_i^*)^{1-a_{ij}^M}(V_j)=0,   $$  which implies that $R_i(M)$ admits $i$-th reflection and $ a_{ij}^M=a_{ij}^{R_i(M)}$ for all $j \neq i$.
    It is clear that the $i$-th position of $R_i^2(M)$ is isomorphic to $M_i$. For $j \neq i$, 
    $$ \mathcal{F}_0\left(\operatorname{ad}\bB(M_i^*)(V_j)\right)=\mathcal{F}_0\left(\Omega_i\left(\operatorname{ad}\left(\bB\left(M_i\right)(M_j)\right)\right)\right)=\mathcal{F}^0(\operatorname{ad}(\bB(M_i))(M_j))=M_j.$$
    Therefore $R_i^2(M)_j \cong M_j$ for all $j \neq i$, and hence 
    $R_i^2(M)\cong M$.
\end{proof}

\begin{thm}\label{C-thm5.9}
   Let  $M = (M_1, \ldots, M_{\theta}) \in \mathcal{F}_{\theta}$ such that each component $M_i$ is irreducible. Suppose $M$ admits all reflections. We define the set $$\mathcal{X}=\{ [P]\mid P \in \mathcal{F}_{\theta}(M) \},$$ and the map  $$r: \mathbb{I} \times \mathcal{X} \rightarrow \mathcal{X},\ i \times [X] \mapsto [R_i(X)].$$ Furthermore, we assume each $\mathcal{B}(P_i)$ is finite-dimensional for $P \in \mathcal{X}$ and $i \in \mathbb{I}$, then 
    $$ \mathcal{G}(M)=(\mathbb{I},\mathcal{X},r,(A^X)_{X\in \mathcal{X}}),$$
    where $A^{[X]}=(a_{ij}^X)_{i,j \in \mathbb{I}}$ for all $[X] \in \mathcal{X}$, is a semi-Cartan graph.\end{thm}
 \begin{proof}
     It follows by definition of semi-Cartan graph and Corollary \ref{C-cor5.8}.
 \end{proof}
\section{Applications}
Building upon the framework developed in the preceding sections, we now present some applications of the semi-Cartan graph theory to  Nichols algebras over certain coquasi-Hopf algebras.  We establish several criteria that connect the finite-dimensionality of a Nichols algebra to the finiteness and structure of its associated semi-Cartan graph and Weyl groupoid. Then we apply this theory to provide a new proof for the key result in [\citealp{huang2024classification}].
\subsection{ A criterion for finite-dimensional Nichols algebras}

We still work on the setting of $\GG$, where $G$ is a finite group and $\Phi$ is a $3$-cocycle on $G$.
Now fix a positive integer $\theta$ and  denote the index set $\mathbb{I}=\{ 1,2,...,\theta\}$. 
  Let  $M = (M_1, \ldots, M_{\theta}) \in \mathcal{F}_{\theta}$. We endow a $\mathbb{Z}^{\theta}$-grading on $\mathcal{B}(M)$ by setting
 $$ \operatorname{deg}(M_i)=\alpha_i, \ 1\leq i \leq \theta,$$ where $\alpha_i$, $1\leq i \leq \theta$ are the standard basis of $\mathbb{Z}^{\theta}$. For each $\beta \in \mathbb{Z}^{\theta}$, define the homogeneous component $\mathcal{B}(M)_{\beta}=\{x \in \bB(M)\mid \operatorname{deg}(x)=\beta \}$,
 and the support of $\bB(M) $ as 
 $$\operatorname{Supp}\bB(M):=\{ \beta \in \mathbb{Z}^{\theta}\mid \bB(M)_{\beta}\neq 0 \}.$$
 \begin{lemma} For each $i \in \mathbb{I}$, suppose $M$ admits $i$-th reflection,
     \begin{align}\label{C-4.1}
         \operatorname{Supp}\bB(M)\cup \operatorname{Supp}\bB(M^*)&=\operatorname{Supp}\bB(R_i(M))\cup \operatorname{Supp}\bB(R_i(M)^*),\\
         s_i^M(\operatorname{Supp}\bB(M)\cup \operatorname{Supp}\bB(M^*))&=\operatorname{Supp}\bB(R_i(M))\cup \operatorname{Supp}\bB(R_i(M)^*).\label{C-4.2}
     \end{align}
 \end{lemma}

\begin{proof}
    If we regard $ \mathcal{B}(M)$ as an $\mathbb{N}_0$-graded object in  $\GG$, by [\citealp{defofnichols}, Lemma 2.6] the graded dual of $\bB(M)$ is isomorphic to $\bB(M^*)$. We thus obtain an isomorphism of $\mathbb{N}_0$-graded objects in  $\GG$:
    $$\bB(M) \otimes \bB(M^*)\cong K_i \otimes \bB(M_i) \otimes K_i^{\operatorname{gr-dual}}  \otimes \bB(M_i^*).$$
    Similarly, for the reflected tuple, we have
    $$ \bB(R_i(M)) \otimes \bB(R_i(M)^*)\cong K_i \otimes \bB(M_i^*) \otimes  K_i^{\operatorname{gr-dual}} \otimes \bB(M_i).$$
    Since $A \otimes B \cong B \otimes A$ as $\mathbb{Z}^{\theta}$-graded vector spaces, the supports on both sides coincide, which establishes the first claim. \par 
   The second claim follows directly from the definition of $s_i^M$.
\end{proof}
Next, we show that reflections preserve the dimension of Nichols algebras.
\begin{lemma}\label{C-lem4.2}
    Suppose $\operatorname{dim}\bB(M)< \infty$, then $M$ gives rise to a semi-Cartan graph, and $\operatorname{dim}\bB(R_i(M))=\operatorname{dim}\bB(M)$ for each $i \in \mathbb{I}$.
\end{lemma}
\begin{proof}
    For each $i\in \mathbb{I}$, as vector spaces, we have the isomorphism $\bB(M)\cong K_i \otimes \bB(M_i)$. Therefore $K_i$ is finite-dimensional, and $\operatorname{ad}(\bB(M_i))(M_j)$ is finite-dimensional for each $j \in \mathbb{I}$. It follows by definition that $M$ admits the $i$-th reflection. Applying the  equivalence $\Omega_i$, we obtain $\Omega_i(K_i)=K_i \in \AMID$  and  $\bB(R_i(M))\cong K_i \otimes \bB(M_i^*)$ as vector space. Since $\bB(M_i)$ is finite-dimensional, we have $\operatorname{dim}\bB(M_i)=\operatorname{dim}\bB(M_i^*)$. Therefore $\operatorname{dim}\bB(R_i(M))=\operatorname{dim}\bB(M)$ for each $i \in \mathbb{I}$. By iterated reflections, $M$ admits all reflections and gives rise to a semi-Cartan graph.
\end{proof}
Before we state the next lemma, we recall the definition of a finite semi-Cartan graph.
\begin{definition}
   Let $\mathcal{G}=(\mathbb{I},\mathcal{X},r,(A^X)_{X\in \mathcal{X}})$ be  a semi-Cartan graph. For all $X \in \mathcal{X}$, the set
   $$ \Delta^{X,\operatorname{re}}=\{\omega(\alpha_i)\in \mathbb{Z}^{\mathbb{I}}\mid \omega \in \operatorname{Hom}(\mathcal{W}(\mathcal{G}),X), i \in \mathbb{I} \}$$
   is called the set of real roots of $\mathcal{G}$ at $X$. We call the semi-Cartan graph $\mathcal{G}$  finite, if $\Delta^{X,\operatorname{re}}$ is finite for all $X \in \mathcal{X}$.
\end{definition}
We now relate the finiteness of the Nichols algebra to the structure of the associated semi-Cartan graph.
\begin{lemma}\label{C-lem6.3}
Suppose $\bB(M)$ is finite-dimensional, then the semi-Cartan graph $\mathcal{G}(M)$ is finite.
\end{lemma}

\begin{proof}
Let  $W= M_1 \oplus \cdots \oplus M_\theta$. Clearly, the homogeneous degrees of elements in $W$ lie in the union $W \subseteq \operatorname{Supp}\bB(M)\cup \operatorname{Supp}\bB(M^*)$. For $1 \leq j \leq \theta $, we apply equations  (\ref{C-4.1}),(\ref{C-4.2}),  together with the fact that  $R_i^2(M)\cong M$ to obtain
    \begin{align*}
        s_i^{R_i(M)}(\alpha_j)&\subseteq s_i^{R_i(M)}(\operatorname{Supp}\bB(M)\cup \operatorname{Supp}\bB(M^*)) \\
        &=s_i^{R_i(M)}(\operatorname{Supp}\bB(R_i(M))\cup \operatorname{Supp}\bB(R_i(M^*)))\\
        &=\operatorname{Supp}\bB(M)\cup \operatorname{Supp}\bB(M^*).
    \end{align*}
   We recall that every element $\omega \in \operatorname{Hom}(\mathcal{W}(\mathcal{G}), M)$ can be expressed as an iteration of simple reflections. It follows that  $\Delta^{M \operatorname{re}} \subseteq \operatorname{Supp}\bB(M)\cup \operatorname{Supp}\bB(M^*)$.
   Now, since $\mathcal{B}(M)$ is finite-dimensional, the set $\operatorname{Supp} \mathcal{B}(M)$ is finite. Moreover, as vector spaces, $\mathcal{B}(M) \cong \mathcal{B}(M^*)^{\mathrm{gr\text{-}dual}}$, so $\operatorname{Supp} \mathcal{B}(M^*)$ is also finite. Combining these observations, we conclude that $\Delta^{M,\operatorname{re}}$ is finite. The same reasoning applies to any $N \in X$: we have 
    $$ \Delta^{N, \operatorname{re}} \subseteq \operatorname{Supp}\bB(N)\cup \operatorname{Supp}\bB(N^*).$$
 Then $\Delta^{N, \operatorname{re}}$ is  again finite by Lemma \ref{C-lem4.2}. Therefore, $\mathcal{G}(M)$ is a finite semi-Cartan graph.
    \end{proof}
    \begin{lemma}\textup{[\citealp{HS10}, Lemma 5.1]}\label{C-lem4.4}
     Suppose  $\mathcal{G}=(\mathbb{I},\mathcal{X},r,(A^X)_{X\in \mathcal{X}})$ is a semi-Cartan graph. Then $\mathcal{G}$ is finite semi-Cartan graph if and only if the Weyl groupoid $\mathcal{W}(\mathcal{G})$ is finite.
    \end{lemma}
\begin{definition}
    Suppose $\mathcal{G}=(\mathbb{I},\mathcal{X},r,(A^X)_{X\in \mathcal{X}})$ is a semi-Cartan graph. If $A^X=A^Y$ for all $X,Y \in \mathcal{X}$, then $\mathcal{G}$ is called a standard semi-Cartan graph.
\end{definition}
The proof of the next theorem is similar to the Hopf algebra situation; we include it here for completeness.
\begin{thm}\label{C-thm4.6}
 Let $M = (M_1, \ldots, M_{\theta}) \in \mathcal{F}_{\theta}$, where each $M_i$ is irreducible for all $i \in \mathbb{I}$. If $\operatorname{dim}\bB(M) < \infty$ and $\mathcal{G}(M)$ is standard, then $A^M$ must be a finite Cartan matrix.
\end{thm}
\begin{proof}
    Since $\operatorname{dim}\bB(M) < \infty$, Lemma \ref{C-lem6.3} implies that $\mathcal{G}(M)$  is a finite Semi-Cartan graph.  By Lemma \ref{C-lem4.4}, $\mathcal{W}(\mathcal{G}(M))$ is finite, in particular, $\operatorname{Hom}(\mathcal{W}(\mathcal{G}(M)))$ is finite.\par 
    Let $W(A^M)$ be the Weyl group of the Kac-Moody Lie algebra of the Cartan matrix $A^M$. Note that 
    $$ W(A^M)=\langle s_i^M \in \operatorname{Aut}(\mathbb{Z}^{\theta})\mid 1\leq i \leq \theta \rangle.$$
    Since $\mathcal{G}(M)$ is standard, we know that
    $$ \operatorname{Hom}(\mathcal{W}(\mathcal{G}(M)))\longrightarrow W(A^M),\ \ (Y,s,X)\mapsto s $$
    is well-defined and surjective. Therefore $W(A^M)$ is finite, which implies that $A^M$ is of finite type.
\end{proof}
\subsection{A class of infinite-dimensional Nichols algebras}
The purpose of this section is to give a new proof of Proposition 4.1 in \cite{huang2024classification} through applying our previous observations.\par 
Let $G=\mathbb{Z}_{m_1}\times \mathbb{Z}_{m_2}\times \mathbb{Z}_{m_3}=\langle g_1\rangle\times \langle g_2\rangle \times \langle g_3\rangle$ with $m_1 \mid m_2 $, $m_2\mid m_3$, and
\begin{align}
	& \Phi\left( g_1^{i_1}  g_2^{i_2}g_3^{i_3}, g_1^{j_1} g_2^{j_2} g_3^{j_3}, g_1^{k_1} g_2^{k_2}g_3^{k_3}\right)  = (-1)^{k_1 j_2 i_3}
\end{align}
be a $3$-cocycle on $G$, where $0 \leq i_l < m_1$, $0 \leq j_l < m_2$, $0\leq k_l < m_3$, $1\leq l \leq 3$.\par
Let $V_1,V_2,V_3 \in \GG $ be simple modules such that $\operatorname{deg}(V_1)=g_1$, $\operatorname{deg}(V_2)=g_2$, $\operatorname{deg}(V_3)=g_3$ and $\operatorname{dim}(V_1)=\operatorname{dim}(V_2)=\operatorname{dim}(V_3)=2$. Moreover, $V_1=\left\{X_1',X_2'\right\},V_2=\left\{Y_1',Y_2'\right\},V_3=\left\{Z_1',Z_2'\right\}$ can be assumed satisfying the following equations:

\begin{equation*}
	\left\{
	\begin{aligned}
		&g_1\rhd X_i'=-X_i', \ i=1,2, \\
		&g_2\rhd X_1'= X_1', g_2\rhd X_2'=-X_2',\\
		&g_3 \rhd X_1'= X_2', g_3 \rhd X_2'= X_1'.\\
	\end{aligned}
	\right.
\quad
	\left\{
	\begin{aligned}
		&g_2\rhd Y_i'=-Y_i', \ i=1,2, \\
		&g_3\rhd Y_1'= Y_1', g_3\rhd Y_2'=-Y_2',\\
		&g_1 \rhd Y_1'= Y_2', g_1 \rhd Y_2'= Y_1',\\
	\end{aligned}
	\right.
\quad
	\left\{
	\begin{aligned}
		&g_3\rhd Z_i'=-Z_i',\ i=1,2, \\
		&g_2\rhd Z_1'= Z_1', g_2\rhd Z_2'=- Z_2',\\
		&g_1 \rhd Z_1'=Z_2', g_1 \rhd Z_2'= Z_1'.\\
	\end{aligned}
	\right.
\end{equation*}


\begin{rmk}\label{C-rmk 4.7}\rm
(i) Note that $m_1,m_2,m_3$ must be even. Indeed, from the relation $X_i=e^{m_1} \rhd X_i=(-1)^{m_1}X_i$, it follows that  $m_1$ is even. Similarly $m_2, m_3$ are even. \par 
(ii) $\bB(V_1)$, $\bB(V_2)$, $\bB(V_3)$ are finite-dimensional Nichols algebras by [\citealp{huang2024classification}, Proposition 2.12]. Moreover, $\bB(V_1\oplus V_2)$, $\bB(V_1 \oplus V_3)$, $
\bB(V_2\oplus V_3)$ are finite-dimensional by [\citealp{nondiagonal}, Proposition 5.1].\par
(iii)  The requirement $\operatorname{dim}(V_1)=\operatorname{dim}(V_2)=\operatorname{dim}(V_3)=2$ is in fact necessary, otherwise the corresponding Nichols algebra must be infinite-dimensional.
\end{rmk}
\begin{thm}\label{C-thm 4.8}\textup{[\citealp{huang2024classification}, Proposition 4.1]}
    With above notations, $\mathcal{B}(V_1\oplus V_2 \oplus V_3)$ is infinite-dimensional.
\end{thm}
It should be emphasized that the above theorem  plays a key role in that paper and the original proof relies on heavy computations
The purpose of the following discussion is to give a new proof of this theorem by applying the reflection theory of Nichols algebras.  \par 
To simplify the proof, we proceed with the following reduction. Let $H=\mathbb{Z}_2 \times \mathbb{Z}_2 \times \mathbb{Z}_2=\langle h_1\rangle\times \langle h_2\rangle \times \langle h_3\rangle$, and 
$$\phi\left( h_1^{i_1}  h_2^{i_2}h_3^{i_3}, h_1^{j_1} h_2^{j_2} h_3^{j_3}, h_1^{k_1} h_2^{k_2}h_3^{k_3}\right)  = (-1)^{k_1 j_2 i_3},$$
be a $3$-cocycle on $H$, where $0\leq i_l,j_l,k_l\leq 1$, $1\leq l \leq 3$.
By [\citealp{huang2024classification}, Lemma 3.5], we may give a complete list of irreducible Yetter-Drinfeld modules.  In particular, restricting to those irreducible Yetter–Drinfeld modules that generate finite-dimensional Nichols algebras, the isomorphism classes are given by the set
\begin{equation}
    S=\{ [^hW]\mid \ ^hW \ \text{is irreducible},\  h \in \mathbb{Z}_2 \times \mathbb{Z}_2 \times \mathbb{Z}_2-\{1,h_1h_2h_3\}\}.
\end{equation}
There are exactly six isomorphism classes. Here $^hW $ represents the simple Yetter–Drinfeld module has comodule structure
$$^hW:=\lbrace w\in W\mid \delta_W(w)=h\otimes w\rbrace.$$ 
We now choose explicit representatives of $S$. For $1\leq i \leq 6$, let $W_i\in S $ be pairwise non-isomorphic simple modules such that $\operatorname{dim}(W_i)=2$, and  $\operatorname{deg}(W_1)=h_1$, $\operatorname{deg}(W_2)=h_2$, $\operatorname{deg}(W_3)=h_3$, $\operatorname{deg}(W_4)=h_1h_2$, $\operatorname{deg}(W_5)=h_1h_3$, $\operatorname{deg}(W_6)=h_2h_3$. Moreover, we may take $$\begin{aligned}
W_1 &= \{X_1, X_2\}, & W_2 &= \{Y_1, Y_2\}, & W_3 &= \{Z_1, Z_2\}, \\
W_4 &= \{R_1, R_2\}, & W_5 &= \{S_1, S_2\}, & W_6 &= \{T_1, T_2\},
\end{aligned}$$ satisfying the following relations:

\begin{equation*}
	\left\{
	\begin{aligned}
		&h_1\rhd X_i=-X_i, \ i=1,2, \\
		&h_2\rhd X_1= X_1, h_2\rhd X_2=-X_2,\\
		&h_3 \rhd X_1= X_2, h_3 \rhd X_2= X_1.\\
	\end{aligned}
	\right.
\quad
	\left\{
	\begin{aligned}
		&h_2\rhd Y_i=-Y_i, \ i=1,2, \\
		&h_3\rhd Y_1= Y_1, h_3\rhd Y_2=-Y_2,\\
		&h_1 \rhd Y_1= Y_2, h_1 \rhd Y_2= Y_1.\\
	\end{aligned}
	\right.
\quad
	\left\{
	\begin{aligned}
		&h_3\rhd Z_i=-Z_i,\ i=1,2, \\
		&h_2\rhd Z_1= Z_1, h_2\rhd Z_2=- Z_2,\\
		&h_1 \rhd Z_1=Z_2, h_1 \rhd Z_2= Z_1.\\
	\end{aligned}
	\right.
\end{equation*}
\begin{equation*}
	\left\{
	\begin{aligned}
		&(h_1h_2)\rhd R_i=-R_i, \ i=1,2, \\
		&h_1\rhd R_1= R_1, h_1\rhd R_2=-R_2,\\
		&h_3 \rhd R_1= R_2, h_3 \rhd R_2= R_1.\\
	\end{aligned}
	\right.
\quad
	\left\{
	\begin{aligned}
		&(h_1h_3)\rhd T_i=-T_i, \ i=1,2, \\
		&h_1\rhd T_1= T_1, h_1\rhd T_2=-T_2,\\
		&h_2 \rhd T_1= T_2, h_1 \rhd T_2= T_1.\\
	\end{aligned}
	\right.
\quad
	\left\{
	\begin{aligned}
		&(h_2h_3)\rhd S_i=-S_i,\ i=1,2, \\
		&h_2\rhd S_1= S_1, h_2\rhd S_2=-S_2,\\
		&h_1 \rhd S_1=S_2, h_1 \rhd S_2= S_1.\\
	\end{aligned}
	\right.
 \end{equation*}
 \begin{rmk}\rm
  The category $\HHp_{\operatorname{fd}}$ is rigid. It is straightforward to verify that for each $1\leq i\leq 6$,  $W_i^*$ does exist  with $W_i^*\cong W_i$. This implies $W_i^* \in S$ as well. Moreover, $\mathcal{B}(W_i)\cong \mathcal{B}(W_i^*)$ as Hopf algebras in $\HHp$.
 \end{rmk}
 The application of semi-Cartan graph theory is the following result.
\begin{thm}\label{C-thm 4.10}
    Let $W=(W_1,W_2,W_3)$ be the $3$-tuple, then $W$ admits all reflections and $\mathcal{G}(W)$ is a standard semi-Cartan graph.
\end{thm}
To establish this result, we first prove the following key lemma:
\begin{lemma}\label{C-lem6.10}
    For all $1\leq i\neq j \leq 6$, if $\operatorname{deg}(W_i)\cdot \operatorname{deg}(W_j) \neq 1,h_1h_2h_3$, then
    $\operatorname{ad}(W_i)^2(W_j)=0.$ Moreover, $\operatorname{ad}(W_i)(W_j) \in S$, and as Yetter-Drinfeld module in $\HHp$, we have
    \begin{equation}
        \operatorname{ad}(W_i)(W_j)\cong \operatorname{ad}(W_j)(W_i).
    \end{equation}
    We list the following isomorphism for further use
    \begin{equation}\label{C-6.9}
     \begin{aligned}
        & \operatorname{ad}(W_1)(W_2)\cong  W_4, \operatorname{ad}(W_1)(W_3)\cong  W_5,\operatorname{ad}(W_1)(W_4)\cong W_2, \\ &\operatorname{ad}(W_1)(W_5)\cong W_3,
        \operatorname{ad}(W_2)(W_3)\cong   W_6, \operatorname{ad}(W_2)(W_4)\cong  W_1, \\
        &\operatorname{ad}(W_2)(W_6)\cong W_3,
        \operatorname{ad}(W_3)(W_5)\cong   W_2, \operatorname{ad}(W_3)(W_6)\cong  W_2 \\
        &\operatorname{ad}(W_4)(W_5)\cong W_6,
        \operatorname{ad}(W_4)(W_6)\cong   W_5, \operatorname{ad}(W_5)(W_6)\cong  W_4. 
     \end{aligned}         
     \end{equation}
\end{lemma}
\begin{proof}
     This proof involves extensive but analogous computations. 
   
  \textbf{Step1}\par
     We begin by establishing that $\operatorname{ad}(W_1)^2(W_2)=0.$ To this end, we compute the coproducts:
	$$
	\begin{aligned}
		& \Delta\left(\operatorname{ad}(X_1)\left(Y_1\right)\right)=1 \otimes \operatorname{ad}({X_1})\left(Y_1\right)+\operatorname{ad}({X_1})\left(Y_1\right) \otimes 1+X_1 \otimes\left(Y_1+Y_2\right), \\
		& \Delta\left(\operatorname{ad}({X_1})\left(Y_2\right)\right)=1 \otimes \operatorname{ad}({X_1})\left(Y_2\right)+\operatorname{ad}({X_1})\left(Y_2\right) \otimes 1+X_1 \otimes\left(Y_1+Y_2\right), \\
		& \Delta\left(\operatorname{ad}({X_2})\left(Y_1\right)\right)=1 \otimes \operatorname{ad}({X_2})\left(Y_1\right)+\operatorname{ad}({X_2})\left(Y_1\right) \otimes 1+X_2 \otimes\left(Y_1-Y_2\right), \\
		& \Delta\left(\operatorname{ad}({X_2})\left(Y_2\right)\right)=1 \otimes \operatorname{ad}({X_2})\left(Y_2\right)+\operatorname{ad}({X_2})\left(Y_2\right) \otimes 1+X_2 \otimes\left(Y_2-Y_1\right) .
	\end{aligned}
	$$
	From these we deduce the relations: 
	\begin{equation}\label{C-4.5}
\begin{aligned}
	& \operatorname{ad}({X_1})\left(Y_1\right)-\operatorname{ad}({X_1})\left(Y_2\right)=0,\\
	& \operatorname{ad}({X_2})\left(Y_1\right)+\operatorname{ad}({X_2})\left(Y_2\right)=0.
\end{aligned}
	\end{equation}
Therefore $\operatorname{ad}(W_1)(W_2)=\operatorname{span}\{ \operatorname{ad}(X_1)(Y_1),\operatorname{ad}(X_2)(Y_1) \}$.
 Now since 	$X_1^2=0$, we have 
	$$	\operatorname{ad}(X_1)^2\left(Y_1\right)=X_1(X_1Y_1-Y_2X_1)-h_1 \rhd ((X_1Y_1-Y_2X_1))X_1=0 $$
	which implies $\operatorname{ad}(X_1)^2\left(Y_1\right)=\operatorname{ad}(X_1)^2\left(Y_2\right)=0$. 
	Meanwhile, $$\operatorname{ad}({X_2})\left(\operatorname{ad}({X_1})\left(Y_1\right)\right)= \operatorname{ad}({X_2})\left(\operatorname{ad}({X_1})\left(Y_2\right)\right)=0.$$ Since
	$$\begin{aligned}
	\Delta(\operatorname{ad}({X_2})\left(\operatorname{ad}({X_1})\left(Y_1\right)\right))&=(1\otimes X_2+X_2\otimes 1)(1 \otimes \operatorname{ad}({X_1})\left(Y_1\right)+\operatorname{ad}({X_1})\left(Y_1\right) \otimes 1+X_1 \otimes\left(Y_1+Y_2\right)\\&-(1 \otimes \operatorname{ad}({X_1})\left(Y_1\right)+\operatorname{ad}({X_1})\left(Y_1\right) \otimes 1+X_1 \otimes\left(Y_1+Y_2\right))(1\otimes X_2+X_2\otimes 1)\\
	&=1 \otimes\operatorname{ad}({X_2})\left(\operatorname{ad}({X_1})\left(Y_1\right)\right)+\operatorname{ad}({X_2})\left(\operatorname{ad}({X_1})\left(Y_1\right)\right) \otimes 1
	\end{aligned},$$
where we have used equation (\ref{C-4.5}). A similar argument shows $	\operatorname{ad}(X_2)^2\left(Y_1\right)=0$ and \\$\operatorname{ad}({X_1})\left(\operatorname{ad}({X_2})\left(Y_1\right)\right)=0$. We conclude that  $\operatorname{ad}(W_1)^2(W_2)=0$.
\par 
By Lemma \ref{C-lem3.13},  the object
$\operatorname{ad}(W_1)(W_2)\in \HHp$ is simple. Since $\operatorname{deg}(\operatorname{ad}(W_1)(W_2))=h_1h_2$, $\operatorname{ad}(W_1)(W_2)\cong W_4$. Furthermore, by proof of Lemma \ref{C-lem3.15}, we have an isomorphism of  $\mathbb{N}_0$-graded objects
$\operatorname{ad}\bB(W_1)(W_4)\cong \operatorname{ad}\bB(W_1)(\operatorname{ad}(W_1)(W_2))
$. Since $\operatorname{ad}(W_1)^2(W_2)=0$, we have $\operatorname{ad}(W_1)^2(W_4)=0$ since it has same degree as $\operatorname{ad}(W_1)^2(W_2)$ in the standard $\mathbb{N}_0$-grading.
Therefore, $\operatorname{ad}(W_1)(W_4)$ is irreducible and one readily verifies that $\operatorname{ad}(W_1)(W_4)\cong W_2$ in $\HHp$. \par 
Analogous computations yield:
$$ \operatorname{ad}(W_2)^2(W_1)=\operatorname{ad}(W_4)^2(W_1)=0.$$
These identities further imply:
$$ \operatorname{ad}(W_2)(W_1) \cong W_4, \ \operatorname{ad}(W_4)(W_1)\cong W_2,$$
and
\begin{align*}
    &\operatorname{ad}(W_2)^2(W_4)=\operatorname{ad}(W_4)^2(W_2)=0,\\
    &\operatorname{ad}(W_2)(W_4) \cong W_1, \ \operatorname{ad}(W_4)(W_2)\cong W_1.
\end{align*} \par

\textbf{Step2}\par
We now proceed to prove \textbf{$\operatorname{ad}(W_1)^2(W_3)=0$}:
    $$
	\begin{aligned}
		& \Delta\left(\operatorname{ad}({X_1})\left(Z_1\right)\right)=1 \otimes \operatorname{ad}({X_1})\left(Z_1\right)+\operatorname{ad}({X_1})\left(Z_1\right) \otimes 1+X_1 \otimes Z_1+X_2\otimes Z_2, \\
		& \Delta\left(\operatorname{ad}({X_1})\left(Z_2\right)\right)=1 \otimes \operatorname{ad}({X_1})\left(Z_2\right)+\operatorname{ad}({X_1})\left(Z_2\right) \otimes 1+X_1 \otimes Z_2+X_2 \otimes Z_1, \\
		& \Delta\left(\operatorname{ad}({X_2})\left(Z_1\right)\right)=1 \otimes \operatorname{ad}({X_2})\left(Z_1\right)+\operatorname{ad}({X_2})\left(Z_1\right) \otimes 1+X_2 \otimes Z_1 +X_1 \otimes Z_2, \\
		& \Delta\left(\operatorname{ad}({X_2})\left(Z_2\right)\right)=1 \otimes \operatorname{ad}({X_2})\left(Z_2\right)+\operatorname{ad}({X_2})\left(Z_2\right) \otimes 1+X_2 \otimes Z_2+X_1\otimes Z_1 .
	\end{aligned}
	$$
It is straightforward to observe that:
	\begin{equation}\label{C-6.10}
\begin{aligned}
	& \operatorname{ad}({X_1})\left(Z_1\right)-\operatorname{ad}({X_2})\left(Z_2\right)=0, \\
	& \operatorname{ad}({X_2})\left(Z_1\right)+\operatorname{ad}({X_1})\left(Z_2\right)=0.
\end{aligned}
	\end{equation}
 Hence $\operatorname{ad}(W_1)(W_3)=\operatorname{span}\{ \operatorname{ad}(X_1)(Z_1),\operatorname{ad}(X_2)(Z_1) \}$.
 By $X_1^2=X_2^2=0$,
	\begin{align*}
	    \operatorname{ad}(X_1)^2\left(Z_1\right)&=X_1(X_1Z_1-Z_2X_1)-h_1 \rhd (X_1Z_1-Z_2X_1)X_1=0,\\
     \operatorname{ad}(X_1)^2(Z_2)&=X_1(X_1Z_2-Z_1X_1)-h_1 \rhd (X_1Z_2-Z_1X_1)X_1=0,\\
       \operatorname{ad}(X_2)^2\left(Z_1\right)&=X_2(X_2Z_1-Z_2X_2)-h_1 \rhd (X_2Z_1-Z_2X_2)X_2=0,\\
     \operatorname{ad}(X_2)^2(Z_2)&=X_2(X_2Z_2-Z_1X_2)-h_1 \rhd (X_2Z_2-Z_1X_2)X_2=0.
	\end{align*} 
 This argument, combined with (\ref{C-6.10}), implies 
	$\operatorname{ad}(W_1)^2(W_3)=0$. 
 As in the first case, we can derive the following identities:
 \begin{align*}
     \operatorname{ad}(W_3)^2(W_1)=&\operatorname{ad}(W_1)^2(W_5)=\operatorname{ad}(W_3)^2(W_5)=\operatorname{ad}(W_5)^2(W_1)=\operatorname{ad}(W_5)^2(W_3)=0.\\
     &\operatorname{ad}(W_1)(W_3)\cong \operatorname{ad}(W_3)(W_1)\cong W_5,\\ &\operatorname{ad}(W_1)(W_5)\cong \operatorname{ad}(W_5)(W_1)\cong W_3, \\ &\operatorname{ad}(W_3)(W_5) \cong \operatorname{ad}(W_5)(W_3)\cong W_1.
 \end{align*}
\par 
\textbf{Step3}\par Now we turn to proving  $\operatorname{ad}(W_2)^2(W_3)=0$:
$$
	\begin{aligned}
		& \Delta\left(\operatorname{ad}({Y_1})\left(Z_1\right)\right)=1 \otimes \operatorname{ad}({Y_1})\left(Z_1\right)+\operatorname{ad}({Y_1})\left(Z_1\right) \otimes 1, \\
		& \Delta\left(\operatorname{ad}({Y_1})\left(Z_2\right)\right)=1 \otimes \operatorname{ad}({Y_1})\left(Z_2\right)+\operatorname{ad}({Y_1})\left(Z_2\right) \otimes 1+2Y_1 \otimes Z_2, \\
		& \Delta\left(\operatorname{ad}({Y_2})\left(Z_1\right)\right)=1 \otimes \operatorname{ad}({Y_2})\left(Z_1\right)+\operatorname{ad}({Y_2})\left(Z_1\right) \otimes 1+2Y_2 \otimes Z_1, \\
		& \Delta\left(\operatorname{ad}({Y_2})\left(Z_2\right)\right)=1 \otimes \operatorname{ad}({Y_2})\left(Z_2\right)+\operatorname{ad}({Y_2})\left(Z_2\right) \otimes 1.
	\end{aligned}
	$$
Then $\operatorname{ad}(Y_1)(Z_1)=\operatorname{ad}(Y_2)(Z_2)=0$ and 
$\operatorname{ad}(W_2)(W_3)=\{\operatorname{ad}(Y_1)(Z_2), \operatorname{ad}(Y_2)(Z_1)\}.$ By $Y_1^2=Y_2^2=0$, 
\begin{align*}
   \operatorname{ad}(Y_1)^2(Z_2)&=Y_1(Y_1Z_2+Z_2Y_1)-h_2 \rhd (Y_1Z_2+Z_2Y_1)Y_1=0,\\
       \operatorname{ad}(Y_2)^2\left(Z_1\right)&=Y_2(Y_2Z_1-Z_1Y_2)-h_2 \rhd (Y_2Z_1-Z_1Y_2)Y_2=0.
\end{align*}
On the other hand, since $\operatorname{ad}Y_2(Z_2)=0,$ and $Y_1Y_2+Y_2Y_1=0$. The following identity holds.
 $$\begin{aligned}
	\Delta(\operatorname{ad}({Y_2})\left(\operatorname{ad}({Y_1})\left(Z_2\right)\right))&=(1\otimes Y_2+Y_2\otimes 1)(1 \otimes \operatorname{ad}({Y_1})\left(Z_2\right)+\operatorname{ad}({Y_1})\left(Z_2\right) \otimes 1+2Y_1 \otimes Z_2)\\&-(1 \otimes \operatorname{ad}({Y_1})\left(Z_2\right)+\operatorname{ad}({Y_1})\left(Z_2\right) \otimes 1+2Y_1 \otimes Z_2)(1\otimes Y_2+Y_2\otimes 1)\\
	&=1 \otimes\operatorname{ad}({X_2})\left(\operatorname{ad}({X_1})\left(Y_1\right)\right)+\operatorname{ad}({X_2})\left(\operatorname{ad}({X_1})\left(Y_1\right)\right) \otimes 1.
	\end{aligned}$$
Similarly, $\operatorname{ad}(Y_1)(\operatorname{ad}(Y_2)(Z_1))=0$. Hence $\operatorname{ad}(W_2)^2(W_3)=0$.
 We list all identities which can be obtained in this case.
 \begin{align*}
     \operatorname{ad}(W_3)^2(W_2)=&\operatorname{ad}(W_2)^2(W_6)=\operatorname{ad}(W_6)^2(W_2)=\operatorname{ad}(W_3)^2(W_6)=\operatorname{ad}(W_6)^2(W_3)=0.\\
     &\operatorname{ad}(W_2)(W_3)\cong \operatorname{ad}(W_3)(W_2)\cong W_6, \\ &\operatorname{ad}(W_2)(W_6)\cong \operatorname{ad}(W_6)(W_2)\cong W_3, \\ &\operatorname{ad}(W_3)(W_6) \cong \operatorname{ad}(W_6)(W_3)\cong W_2.
 \end{align*}\par
 \textbf{Step4}\par
 The final case  to confirm is $\operatorname{ad}(W_4)^2(W_5)=0$:
$$
	\begin{aligned}
		& \Delta\left(\operatorname{ad}({R_1})\left(T_1\right)\right)=1 \otimes \operatorname{ad}({R_1})\left(T_1\right)+\operatorname{ad}({R_1})\left(T_1\right) \otimes 1+R_1\otimes T_1-R_2\otimes T_2, \\
		& \Delta\left(\operatorname{ad}({R_1})\left(T_2\right)\right)=1 \otimes \operatorname{ad}({R_1})\left(T_2\right)+\operatorname{ad}({R_1})\left(T_2\right) \otimes 1+R_1\otimes T_2+R_2\otimes T_1, \\
		& \Delta\left(\operatorname{ad}({R_2})\left(T_1\right)\right)=1 \otimes \operatorname{ad}({R_2})\left(T_1\right)+\operatorname{ad}({R_2})\left(T_1\right) \otimes 1+R_2\otimes T_1+R_1\otimes T_2, \\
		& \Delta\left(\operatorname{ad}({R_2})\left(T_2\right)\right)=1 \otimes \operatorname{ad}({R_2})\left(T_2\right)+\operatorname{ad}({R_2})\left(T_2\right) \otimes 1+R_2\otimes T_2-R_1\otimes T_1.
	\end{aligned}
	$$
 Then it is direct to see that
	\begin{equation}\label{C-6.95}
\begin{aligned}
	& \operatorname{ad}({R_1})\left(T_1\right)+\operatorname{ad}({T_2})\left(R_2\right)=0, \\
	& \operatorname{ad}({R_2})\left(T_1\right)-\operatorname{ad}({R_1})\left(T_2\right)=0.
\end{aligned}
	\end{equation}
 Hence $\operatorname{ad}(W_4)(W_5)=\operatorname{span}\{ \operatorname{ad}(R_1)(T_1),\operatorname{ad}(R_2)(T_1) \}$.
 By $R_1^2=R_2^2=0$,
	\begin{align*}
	    \operatorname{ad}(R_1)^2\left(T_1\right)&=R_1(R_1T_1+T_2R_1)-(h_1h_2) \rhd (R_1T_1+T_2R_1)R_1=0,\\
     \operatorname{ad}(R_1)^2(T_2)&=R_1(R_1T_2-T_1R_1)-(h_1h_2) \rhd (R_1T_2-T_1R_1)R_1=0,\\
       \operatorname{ad}(R_2)^2\left(T_1\right)&=R_2(R_2T_1+T_2R_2)-(h_1h_2) \rhd (R_2T_1+T_2R_2)R_2=0,\\
     \operatorname{ad}(R_2)^2(T_2)&=R_2(R_2T_2-T_1R_2)-(h_1h_2) \rhd (R_2T_2-T_1R_2)R_2=0.
	\end{align*} 
 This combines (\ref{C-6.95}) implies 
	$\operatorname{ad}(W_4)^2(W_5)=0$.\\
\par The following identities  can be obtained in this case.
 \begin{align*}
     \operatorname{ad}(W_4)^2(W_6)=&\operatorname{ad}(W_6)^2(W_4)=\operatorname{ad}(W_6)^2(W_5)=\operatorname{ad}(W_5)^2(W_6)=\operatorname{ad}(W_5)^2(W_4)=0.\\
     &\operatorname{ad}(W_4)(W_5)\cong \operatorname{ad}(W_5)(W_4)\cong W_6, \\ &\operatorname{ad}(W_4)(W_6)\cong \operatorname{ad}(W_6)(W_4)\cong W_5, \\ &\operatorname{ad}(W_5)(W_6) \cong \operatorname{ad}(W_6)(W_5)\cong W_4.
 \end{align*}\par
We have listed all cases, which proves this lemma.
\end{proof}

To establish that the tuple $W$ admits all reflections, the following structural observation is essential.
\begin{lemma}\label{C-lem6.12}
    Suppose $W$ admits the reflection sequence $(i_1,i_2,..,i_l)$, where $i_1,..,i_l\in \{1,2,3\}$. Then for $j\in \{1,2,3\}$, 
    \begin{equation}
        \operatorname{deg}(R_{i_l}(...(R_{i_1}(W))...)_j) \neq 1,h_1h_2h_3.
    \end{equation} 
    and $$\prod_{j=1}^3\operatorname{deg}(R_{i_l}(...(R_{i_1}(W))...)_j)=h_1h_2h_3.$$
\end{lemma}
\begin{proof}
    We prove it by induction. For the base case $l=1$, Lemma \ref{C-lem3.15} gives,
    $$ R_1(W)\cong (W_1,W_4,W_5), \ R_2(W)\cong (W_4,W_2,W_6), \ R_3(W)\cong (W_5,W_6,W_3).$$
It is straightforward to verify that for each  $1\leq i,j \leq 3$, $\prod_{j=1}^3R_i(W)=h_1h_2h_3$.\par 
    Now assume, for some $l \geq 2$, that $W$ admits the reflection sequence $(i_1,i_2,..,i_{l-1})$, and $$\prod_{j=1}^3\operatorname{deg}(R_{i_{l-1}}(...(R_{i_1}(W))...)_j)=h_1h_2h_3.$$
    Let us denote $\operatorname{deg}(R_{i_{l-1}}(...(R_{i_1}(W))...)_j)=h_j'$ for each $1\leq j \leq 3$, where $h_j' \in \mathbb{Z}_2 \times \mathbb{Z}_2 \times \mathbb{Z}_2$. Then for any $1\leq i_l \leq 3$, we compute
    \begin{align*}
       \prod_{j=1}^3\operatorname{deg}(R_{i_l}(...(R_{i_1}(W))...)_j)&=(\prod_{k\neq i_l}h_k'h_{i_l})h_{i_l}^{-1}\\
       &=h_1h_2h_3,
    \end{align*}
    which establishes the second claim. \par 
Without loss of generality,  Suppose for contradiction that  $\operatorname{deg}(R_{i_l}(...(R_{i_1}(W))...)_1) =h_1h_2h_3$. We must have $\operatorname{deg}(R_{i_l}(...(R_{i_1}(W))...)_2)=\operatorname{deg}(R_{i_l}(...(R_{i_1}(W))...)_3)=h'$ by the second claim. If $i_l=1$, then 
   $$ R_{i_{l-1}}(...(R_{i_1}(W))...)=(h_1h_2h_3,h_1h_2h_3h',h_1h_2h_3h'),$$ which contradicts the assumption. If $i_1=2$, the degrees become $$R_{i_{l-1}}(...(R_{i_1}(W))...)=(h_1h_2h_3h',h',1),$$
   and if $i_1=3$  $$R_{i_{l-1}}(...(R_{i_1}(W))...)=(h_1h_2h_3h',1,h').$$
  Both of which again lead to contradictions.
\par    A similar contradiction arises if we assume $\operatorname{deg}(R_{i_l}(...(R_{i_1}(W))...)_1) =1$,\\ $\operatorname{deg}(R_{i_l}(...(R_{i_1}(W))...)_2)=k'$, $ \operatorname{deg}(R_{i_l}(...(R_{i_1}(W))...)_3)=k''$, where $k'\cdot k''=h_1h_2h_3$. Now if $i_l=1$, we have
   $$ R_{i_{l-1}}(...(R_{i_1}(W))...)=(1,k',k'').$$
   If $i_l=2$, we obtain
    $$R_{i_{l-1}}(...(R_{i_1}(W))...)=(k',k',k'k''),$$
and if $i_l=3$
    $$R_{i_{l-1}}(...(R_{i_1}(W))...)=(k'',k'k'',k'').$$
    All of these cases lead to contradictions. This completes the induction.
\end{proof}

\begin{proof}[Proof of Theorem \ref{C-thm 4.10}]
We proceed by induction to show that $W$ admits all reflection sequences. Without loss of generality, we first consider the $1$st-reflection.
    By proof of Lemma \ref{C-lem6.10}, for all $1\leq j\leq 3$, we have $R_1(W)_j \in S$. Recall that two tuples $W$ and $W'$ are isomorphic if and only if their corresponding components are isomorphic. 
    $R_1(W)\cong (W_1,W_4,W_5)$. By Lemma \ref{C-lem3.15}, 
    $R_i(R_1(W))\cong R_i(W_1,W_4,W_5)$ for $1\leq i \leq 3$. Applying Lemma  \ref{C-lem6.10} again,  we conclude that $(W_1, W_4, W_5)$ admits the $i$-th reflection by analogous reasoning.\par  
Now, suppose W admits the reflection sequence $(i_1, i_2, \ldots, i_{l-1})$, where $1\leq i_1,...,i_{l-1}\leq 3$. 
Then we have $\operatorname{deg}(R_{i_{l-1}}(...(R_{i_1}(W))...)_j)\neq 1,h_1h_2h_3$ for $1\leq j\leq 3$ and $\prod_{j=1}^3\operatorname{deg}(R_{i_{l-1}}(...(R_{i_1}(W))...)_j)=h_1h_2h_3.$
Consequently, for any $1 \leq m < n \leq 3$, the product
$$\operatorname{deg}(R_{i_{l-1}}(...(R_{i_1}(W))...)_m)\cdot \operatorname{deg}(R_{i_{l-1}}(...(R_{i_1}(W))...)_n)$$ is neither $1$ nor $h_1h_2h_3$.
Therefore $W$ admits the $(i_1,i_2,...,i_{l-1},i_l)$-th reflection by Lemma \ref{C-lem6.10}. This completes the induction and shows that $W$ admits all reflection sequences.\par 
For any $N \in \mathcal{F}_3(W)$, we have $\operatorname{deg}(N_i)\neq 1, h_1h_2h_3$ for $1\leq i\leq 3$ by Lemma \ref{C-lem6.12}. Therefore, $\mathcal{B}(N_i)$ is finite-dimensional since $N_i$ belongs to $S$. Thus $W$ will give rise to a semi-Cartan graph $\mathcal{G}(W)$.\par 
Now for all $N \in \mathcal{F}_3(W)$, Lemma \ref{C-lem6.10} implies that the generalized Cartan matrix $A^N$ is given by
$$\begin{pmatrix}
    2 & -1 & -1\\
    -1 &2 & -1 \\
    -1 & -1 &2
\end{pmatrix}$$
Thus 
$\mathcal{G}(W)$ is a standard semi-Cartan graph.
\end{proof}
\begin{cor}
    The Nichols algebra $\mathcal{B}(W)$ is infinite-dimensional.
\end{cor}
\begin{proof}
   $\mathcal{G}(W)$ is a standard semi-Cartan graph by Theorem \ref{C-thm 4.10}.  Suppose, for contradiction, that $\mathcal{B}(W)$ is finite-dimensional, then $A^W$ is a finite Cartan matrix by Theorem \ref{C-thm4.6}. However 
$$A^W=\begin{pmatrix}
    2 & -1 & -1\\
    -1 &2 & -1 \\
    -1 & -1 &2
\end{pmatrix}$$
is not a finite Cartan matrix. This contradiction implies that  $\mathcal{B}(W)$ must be infinite-dimensional.
\end{proof}

\begin{proof}[Proof of Theorem \ref{C-thm 4.8}]
By Remark \ref{C-rmk 4.7}(1), there is a group surjection satisfying
$$ \pi: G \longrightarrow H,\ g_i \mapsto h_i, \ 1\leq i \leq 3.$$
Furthermore, $\pi^*(\phi)=\Phi$.\par 
We construct such a linear map for $1 \leq j \leq 2$, $$F:V_1\oplus V_2 \oplus V_3 \rightarrow W_1 \oplus W_2 \oplus W_3, \ X_j' \mapsto X_j, Y_j'\mapsto Y_j, Z_j' \mapsto Z_j$$
 A direct verification shows that for each $1\leq i \leq 3$,  $v_i \in V_i$ and $g \in G$, the following compatibilities hold:
\begin{align*}
    (\pi \otimes F)\delta_{V_i}(v_i)&=\delta_{W_i}(F(v_i)),\\
    F(g \rhd v_i)&=\pi(g)\rhd F(v_i).
\end{align*}
By [\citealp{QQG}, Lemma 4.4], we deduce that 
$$ \bB(V) \cong \bB(W)$$ in the sense of [\citealp{QQG}, Definition 4.3]. It follows that $\bB(V)$ is infinite-dimensional.    
\end{proof}

\noindent\textbf{Acknowledgment} This work is supported by the National Key R\&D Program of China\\ 2024YFA1013802 and NSFC 12271243.\\[5pt]

	\bibliographystyle{plain}\small
	\bibliography{ref}
\end{document}